\documentclass[12pt]{amsart}
\usepackage{amsmath}
\usepackage{amssymb,amsthm}
\usepackage{amscd}
\usepackage{xypic}

\newtheorem{theorem}{Theorem}[section]
\newtheorem{lemma}[theorem]{Lemma}
\newtheorem{corollary}[theorem]{Corollary}
\newtheorem{proposition}[theorem]{Proposition}

\theoremstyle{definition}
\newtheorem{definition}[theorem]{Definition}

\theoremstyle{remark}

\numberwithin{equation}{section}
\begin{document}
%%%%information about authors%%%%%%%%%%
\title{On the  Cohomology Comparison Theorem\\ }
\subjclass[2000]{Primary: 18, Secondary: 16. }
\author[Alin  Stancu]{\textbf{Alin  Stancu}\\
Department of Mathematics, Columbus State University, Columbus, GA
31907, USA\\}

\email{stancu\_alin1@columbusstate.edu}

\date{}

\begin{abstract}
A relative derived category for the category of modules over a
presheaf of algebras is constructed to identify the relative Yoneda
and Hochschild cohomologies with its homomorphism groups. The
properties of a functor between this category and the relative
derived category of modules over the algebra associated to the
presheaf are studied. We obtain a generalization of the $Special$ $
Cohomology$ $Comparison$ $Theorem$ of M. Gerstenhaber and S. D.
Schack.

\end{abstract}

\maketitle

\newpage

%%%the first section%%%%%%%%
\section{Introduction}
$\;$Hochschild cohomology of a $k$ - algebra $A$, denoted here
$\mathrm{H}^{\bullet}(A,-)$, plays an important role in the study of
associative algebras, by serving as a tool in the deformation theory
of this class of algebras where, broadly speaking, deformations of
$A$ are parameterizations $A_t$, of associative algebras, such that
for $t=0$ one obtains $A$. We mention here only two of its many
other interesting properties: first, separable algebras $A$ are
characterized by $\mathrm{H}^{1}(A,-)=0$ and second, as discovered
by Gerstenhaber, $\mathrm{H}^{\bullet}(A,A)$ has a rich algebraic
structure (of $G$- algebra). In fact, one need not to restrict to a
single algebra and, as M. Gerstenhaber and S. D. Schack did, may
consider deformations of presheaves of algebras, or more general of
diagrams of algebras, where the naturally defined Hochschild
cohomology plays a similar role. The Hochschild cohomology of
presheaves is interesting as a step to subsuming the deformation
theory of complex manifolds in the deformation theory of associative
algebras. The authors mentioned above associated to each presheaf of
algebras $\mathbb{A}$ a single algebra $\mathbb{A}!$ and proved the
$Special $ $Cohomology$ $Comparison$ $Theorem$ which states that
Yoneda and Hochschild cohomologies of the presheaf and the algebra
associated to the presheaf are isomorphic.

Note that Yoneda and Hochschild cohomologies are $\mathbf{relative}$
theories since $k$ is a commutative ring that is not necessarily a
field.

In this paper we develop a relative derived category,
$\mathcal{D}_k^{-} ({\mathbb{A}}-{\mathrm{bimod}})$, of the category
of  bimodules over a presheaf $\mathbb{A}$ of $k$-algebras, one
where the relative Yoneda cohomology,
$\mathrm{Ext}^{i}_{\mathbb{A}-\mathbb{A},k}(\mathbb{M},
\mathbb{N})$, so in particular Hochschild cohomology, can be
regarded as homomorphism groups,
$Mor_{\mathcal{D}_k^{-}({\mathbb{A}}-{\mathrm{bimod}})}
(\mathbb{M}_{\bullet},\mathbb{N}_{\bullet}{[i]})$. The reader should
be aware that the term `presheaf of $k$-algebras' is used to
describe functors $\mathbb{A}$, defined on  posets $\mathcal{C}$,
with images in the category of $k$-algebras.  In this context, we
also show that the functor !, induced between the relative derived
categories of $\mathbb{A}$-bimod and $\mathbb{A!}$-bimod, is full
and faithful and we obtain a generalization of the $Special$
$Cohomology$ $Comparison$ $Theorem$.

This natural construction may be part of providing a more conceptual
interpretation for the Hochschild cohomology of a presheaf of
algebras together with its Gerstenhaber bracket, that of the Lie
algebra of an algebraic group ($i.e$ a group valued functor). In the
case of a single algebra over a field B. Keller, in [6], identifies
$\mathrm{H}^{\bullet}(A,A)$ with the Lie algebra of an algebraic
group by  regarding $\mathrm{H}^{i}(A,A)$ as a homomorphism group
$Mor_{\mathcal{D}({A}-{\mathrm{bimod}})}
({A}^{\bullet},{A}^{\bullet}{[i]})$ in the derived category
$\mathcal{D}({A}-{\mathrm{bimod}})$ and then establishing a
bijection between the latter groups and certain infinitesimal
deformations of $A$ which have a natural Lie bracket. Since the
Gerstenhaber bracket exists on the Hochschild cohomology of
presheaves of algebras presumably a similar interpretation exists
for this situation too. To adapt Keller's technique to this case one
needs to find the ``correct'' derived category that allows the
interpretation of the relative Hochschild cohomology as Hom groups.

$\mathbf{Note:}$ This paper was inspired by [7] and it would have
not been possible without the support of Samuel D. Schack.

%%%%the second section%%%%%%%%%%%%%
\section{Resolutions, adjoint functors and the functor !}
$\;$Let $k$ be a commutative ring and $\mathcal{C}$ a poset viewed
as a category in the usual way: for each $i\leq j$ there is a unique
map $\varphi^{ij}:i\longrightarrow j$. When $A$ is a $k$-algebra and
$M$ any $A$ bimodule we assume $M$ to be symmetric over $k$. ($i.e.$
$ax=xa$ for all $x\in M$ and $a\in k$.) A presheaf of $k$-algebras
over $\mathcal{C}$ is a functor
$\mathbb{A}:\mathcal{C}^{op}\longrightarrow k$-${\mathbf{alg}}$. We
will denote $\mathbb{A}(i)$ by $\mathbb{A}^i$. A presheaf as above
is a special case of functor defined from a small category to the
category of $k$ algebras. In [1] these functors are called a
``diagrams".

The category $\mathbb{A}$-bimod is the category whose objects are
$\mathbb{A}$-bimodules and the maps are maps of bimodules. An
$\mathbb{A}$-bimodule $\mathbb{M}$ is a presheaf of abelian groups
such that $\mathbb{M}^{i}$ is an $\mathbb{A}^{i}$-bimodule
$(\forall) i\in\mathcal{C}$ and for all $i\leq j$ the map
$T_{\mathbb{M}}^{ij}:\mathbb{M}^{j}\longrightarrow\mathbb{M}^{i}$ is
an $\mathbb{A}^{j}$-bimodule map. An $\mathbb{A}$-bimodule map
$\eta:\mathbb{M}\longrightarrow\mathbb{N}$ is a natural
transformation in which $\eta^{i}$ is an $\mathbb{A}^{i}$-bimodule
map $(\forall) i\in\mathcal{C}$.

In defining Yoneda cohomology of the category $\mathbb{A}$-bimod
`allowable' maps play a vital role. A map
$\eta:\mathbb{M}\longrightarrow\mathbb{N}$ is $\mathbf{allowable}$
if $(\forall) i\in\mathcal{C}$ the map
$\eta^{i}:\mathbb{M}^{i}\longrightarrow\mathbb{N}^{i}$ admits a
$k$-bimodule splitting map
$k^{i}:\mathbb{N}^{i}\longrightarrow\mathbb{M}^{i}$ satisfying
$\eta^{i}k^{i}\eta^{i}=\eta^{i}$. We do not require the splitting
maps $k^{i}$ to be natural. An $\mathbb{A}$-bimodule $\mathbb{P}$ is
a $\mathbf{relative}$ $\mathbf{ projective}$ if for every allowable
epimorphism $\mathbb{M}\longrightarrow\mathbb{N}$ the induced map
$Hom_{\mathbb{A}-\mathbb{A}}(\mathbb{P},\mathbb{M})\longrightarrow
Hom_{\mathbb{A}-\mathbb{A}}(\mathbb{P},\mathbb{N})$ is an
epimorphism of sets.

A $\mathbf{relative}$ $\mathbf{projective}$ $\mathbf{allowable}$
$\mathbf{resolution}$ of an $\mathbb{A}$-bimodule $\mathbb{M}$ is an
exact sequence
$\cdots\longrightarrow\mathbb{P}_{n}\cdots\longrightarrow\mathbb{P}_{1}
\longrightarrow\mathbb{P}_{0}\longrightarrow\mathbb{M}\longrightarrow0$
in which all $\mathbb{P}_n$ are relative projective $\mathbb{A}$-
bimodules and all maps are allowable. The category
$\mathbb{A}$-bimod has enough relative projective bimodules and each
bimodule has a relative projective allowable resolution. Moreover,
there is a functorial way of getting this type of resolutions. The
construction of such a resolution is due to M. Gerstenhaber and S.
D. Schack (see [1]) and is based on two facts: First, the
`forgetful' functor
$\mathbb{A}$-bimod$\longrightarrow$$\mathbb{K}$-bimod has a left
adjoint $\mathbb{A}\otimes_{\mathbb{K}} -
\otimes_{\mathbb{K}}\mathbb{A}$, where $\mathbb{K}$ is the constant
presheaf $\mathbb{K}^i=k$,  $ (\forall) i\in\mathcal{C}$. For each
$\mathbb{N}\in\mathbb{A}$-bimod we set
$(\mathbb{A}\otimes_{\mathbb{K}}\mathbb{N}\otimes_{\mathbb{K}}\mathbb{A})^i=
\mathbb{A}^i\otimes_k\mathbb{N}^i\otimes_k\mathbb{A}^i$ and the map
$\mathbb{A}^j\otimes_k\mathbb{N}^j\otimes_k\mathbb{A}^j\longrightarrow
\mathbb{A}^i\otimes_k\mathbb{N}^i\otimes_k\mathbb{A}^i$
corresponding to $i\leq j$ in $\mathcal{C}$ is just
$\varphi^{ij}\otimes T_{\mathbb{N}}^{ij}\otimes\varphi^{ij}$.

The corresponding categorical bar resolution, of [2], of an
$\mathbb{A}$-bimodule $\mathbb{N}$, denoted
$\mathcal{B}_\bullet(\mathbb{N})$, is allowable and since
$\mathcal{B}_{q}(\mathbb{N})=\mathbb{A}\otimes_{\mathbb{K}}\mathcal{B}_{q-1}(\mathbb{N})
\otimes_{\mathbb{K}}\mathbb{A}$ we have that
$\mathcal{B}_{q}(\mathbb{N})^i$ is a relative projective
$\mathbb{A}^i$-bimodule $(\forall) i\in \mathcal{C}$. In addition,
the resolution has a functorial contracting homotopy
$x_q:\mathcal{B}_{q}(\mathbb{N})\longrightarrow\mathcal{B}_{q+1}(\mathbb{N})$,
$x_q(a)=1\otimes a\otimes 1$.

Second, observe that $(\forall) i\in\mathcal{C}$ the functor
$(i)^{*}:\mathbb{A}$-bimod $\longrightarrow\mathbb{A}^{i}$-bimod
defined by $(i)^{*}\mathbb{M}=\mathbb{M}^{i}$ admits a left adjoint
$(i)_{!}:\mathbb{A}^{i}$-bimod$\longrightarrow\mathbb{A}$-bimod,
where $(i_{!}M)^{h}=\mathbb{A}^{h}\otimes_{\mathbb{A}^{i}}M
\otimes_{\mathbb{A}^{i}}\mathbb{A}^{h}$ if $h\leq i $ and
$(i_{!}M)^{h}=0$ otherwise. If $h\leq j\leq i$ the map
$(i_{!}M)^{j}\longrightarrow(i_{!}M)^{h}$ is
$\varphi^{hj}\otimes{Id_M}\otimes\varphi^{hj}$ and it is zero
otherwise.

Combining the functors $(i)^{*}$ we obtain a single exact functor
$\mathcal{R}:\mathbb{A}$-bimod$\longrightarrow\prod_{i\in\mathcal{C}}(\mathbb{A}^{i}$-$
\mathrm{bimod})$, defined on objects by
$\mathcal{R}\mathbb{M}=\prod_{i\in \mathcal{C}}{\mathbb{M}^{i}}$ and
whose left adjoint $\mathcal{L}$ is defined on objects by
$\mathcal{L}{M_i}=\coprod_{i\in\mathcal{C}}(i)_{!}M_i$. Applying
again the categorical bar resolution of [2] we obtain an allowable
resolution with a functorial contracting homotopy. We denote this
resolution by $\mathcal{S_{\bullet}}$. Thus
$\mathcal{S}_p=(\mathcal{LR})^{p+1}=\mathcal{LRS}_{p-1}$ and the
boundary maps $d_p:\mathcal{S}_{p+1}\longrightarrow\mathcal{S}_{p}$
are defined inductively by
$d_p=\varepsilon_{\mathcal{S}_p}-\mathcal{LR}d_{p-1}$, where
$d_{-1}=\varepsilon$ is the counit of the adjunction. The
contracting homotopy is the unit
$\eta_{\mathcal{RS}_p}:\mathcal{RS}_p\longrightarrow\mathcal{RS}_{p+1}$.

Here is a more direct description of $\mathcal{S_{\bullet}}$. Let
[p] be the linearly ordered set $\{0<1<\dots<p\}$. A covariant
functor $\sigma:[p]\rightarrow\mathcal{C}$ is called a $p$-simplex.
Thus $p$-simplices are objects of the functor category
$\mathcal{C}^{[p]}$. The domain of $\sigma$ is defined as
$\sigma(0)$ and is denoted by $d\sigma$. Similarly, the codomain of
$\sigma$ is defined as $\sigma(p)$ and is denoted by $c\sigma$. For
each $p$-simplex $\sigma$ we write $\sigma=(\sigma^{01},
\dots,\sigma^{p-1,p})$ and define
$$\sigma_r=\begin{cases}
(\sigma^{12},\dots,\sigma^{p-1,p}) $\;$ $\;$$\;$$\;$ $\;$$\;$$\;$ $\;$$\;$$\;$$\;$ $\;$$\;$ $\;$$\;$$\;$ $\;$$\;$$\;$ $\;$$\;$$\;$ $\;$$\;$\mathrm{if} $\;$ $\;$$\;$$\;$ $\;$$\;$$r=0$\\
(\sigma^{01},\dots,\sigma^{r-1,r+1},\dots,\sigma^{p-1,p}) $\;$ $\;$$\;$$\;$ \mathrm{if}$\;$ $\;$$\;$ 0<r<p\\
(\sigma^{01},\dots,\sigma^{p-2,p-1})$\;$ $\;$$\;$$\;$ $\;$$\;$$\;$ $\;$$\;$$\;$ $\;$$\;$$\;$ $\;$$\;$$\;$ $\;$$\;$$\;$ $\;$$\;$ \mathrm{if} $\;$ $\;$$\;$$\;$ $\;$$\;$$r=p$\\
\end{cases}$$

Note that $d\sigma_r=d\sigma=\sigma(0)$ if $r\neq0$ and
$d\sigma_0=\sigma(1)$. Similarly, $c\sigma_r=c\sigma=\sigma(p)$ if
$r\neq p$ and $c\sigma_p=\sigma(p-1)$. Also, note that $d\sigma\leq
d\sigma_r$ and $c\sigma_r\leq c\sigma$ and recall that the structure
maps defining presheaves  and bimodules are contravariant.

For $\mathbb{N}\in\mathbb{A}$-$\mathrm{bimod}$ and $p\geq0$ we have
$\mathcal{S}_p\mathbb{N}=\coprod_{\sigma\in\mathcal{C}^{[p]}}\mathcal{S}^{\sigma}_p\mathbb{N}$,
where
$\mathcal{S}^{\sigma}_p\mathbb{N}=(d\sigma)_{!}(\mathbb{A}^{d\sigma}\otimes_{\mathbb{A}
^{c\sigma}}\mathbb{N}^{c\sigma}\otimes_{\mathbb{A}^{c\sigma}}\mathbb{A}^{d\sigma})$
and  $\mathbb{A}^{d\sigma}$ is an $\mathbb{A}^{c\sigma}$-bimodule
via the map
$\varphi^{d\sigma,c\sigma}:\mathbb{A}^{c\sigma}\longrightarrow\mathbb{A}^{d\sigma}$.

For $p\geq0$, the boundary
$\partial:\mathcal{S}_p\mathbb{N}\longrightarrow\mathcal{S}_{p-1}\mathbb{N}$
is a sum $\partial=\sum_{r=0}^{p}(-1)^r\partial_r$ where the
restriction of $\partial_r$ to $\mathcal{S}_p^{\sigma}$ is denoted
$\partial_r^{\sigma}:\mathcal{S}_p^{\sigma}\mathbb{N}=(d\sigma)_{!}(\mathbb{A}^{d\sigma}
\otimes_{\mathbb{A}^{c\sigma}}\mathbb{N}^{c\sigma}\otimes_{\mathbb{A}^{c\sigma}}
\mathbb{A}^{d\sigma})\longrightarrow(d\sigma_r)_{!}
(\mathbb{A}^{d\sigma_r}
\otimes_{\mathbb{A}^{c\sigma_r}}\mathbb{N}^{c\sigma_r}\otimes_{\mathbb{A}^{c\sigma_r}}
\mathbb{A}^{d\sigma_r})=\mathcal{S}_{p-1}^{\sigma_r}\mathbb{N}$.

We obtain that for $h\leq d\sigma$ and
$a\otimes{n}\otimes{a'}\in(\mathcal{S}_p^\sigma\mathbb{N})^h=\mathbb{A}^{h}
\otimes_{\mathbb{A}^{c\sigma}}\mathbb{N}^{c\sigma}
\otimes_{\mathbb{A}^{c\sigma}}\mathbb{A}^h$,
$\partial_r^{\sigma}(a\otimes n\otimes a')=a\otimes
T_{\mathbb{N}}^{c\sigma_r,c\sigma}(n)\otimes
a'\in(\mathcal{S}_{p-1}^{\sigma_r}\mathbb{N})^h$. Here
$T_{\mathbb{N}}^{c\sigma_r,c\sigma}$ is the structure map of the
bimodule $\mathbb{N}$ corresponding to $c\sigma_r\leq c\sigma$. In
particular, when $r=0$ we get $\partial_0^{\sigma}(a\otimes n\otimes
a')=a\otimes n\otimes a'$, and when $r=p$ we get
$\partial_p^{\sigma}(a\otimes n\otimes a')= a\otimes
T_{\mathbb{N}}^{c\sigma_p,c\sigma}(n)\otimes a'$.

The augmentation map
$\varepsilon:\mathcal{S}_0\mathbb{N}=\coprod_{i\in\mathcal{C}}\mathcal{S}_0^i=
\coprod_{i\in\mathcal{C}}(i)_{!}(\mathbb{A}^{i}
\otimes_{\mathbb{A}^i}\mathbb{N}^i
\otimes_{\mathbb{A}^i}\mathbb{A}^i)\longrightarrow\mathbb{N}$ is
defined on the components $(i)_{!}(\mathbb{A}^{i}
\otimes_{\mathbb{A}^i}\mathbb{N}^i
\otimes_{\mathbb{A}^i}\mathbb{A}^i)$. For $h\leq i$,$\;$
$(i)_{!}(\mathbb{A}^{i} \otimes_{\mathbb{A}^i}\mathbb{N}^i
\otimes_{\mathbb{A}^i}\mathbb{A}^i)^h=\mathbb{A}^h\otimes_{\mathbb{A}^i}\mathbb{N}^i\otimes
_{\mathbb{A}^i}\mathbb{A}^h\longrightarrow\mathbb{N}^h$ is given by
$1\otimes n\otimes 1\longrightarrow T_{\mathbb{N}}^{hi}(n)$.

For $i\in\mathcal{C}$ the contracting homotopy
$\kappa_p^i:(\mathcal{S}_p\mathbb{N})^i\longrightarrow(\mathcal{S}_{p+1}\mathbb{N})^i$
is given componentwise by
$(\mathcal{S}_p^{\sigma}\mathbb{N})^i\longrightarrow(\mathcal{S}_{p+1}^{(i,\sigma)}\mathbb{N})^i=
identity$, where $(\mathcal{S}_{p+1}^{(i,\sigma)}\mathbb{N})=0$ if
$i\nleq d\sigma$. If $i\leq d\sigma$, then $(i,\sigma)$ is the
simplex $(i, d\sigma=\sigma(0),\ldots,\sigma(p))$.

In general the above resolution is not a relative projective
resolution, but it is when each $\mathbb{N}^i$ is a relative
projective $\mathbb{A}^i$-bimodule. Thus, to construct a relative
projective allowable resolution of an $\mathbb{A}$-bimodule
$\mathbb{N}$ we take the resolution
$\mathcal{B}_{\bullet}(\mathbb{N})\longrightarrow\mathbb{N}$,
determined by the forgetful functor and its left adjoint, and then
apply $\mathcal{S}_\bullet$ to it to obtain a double complex
$\mathcal{S}_{\bullet}\mathcal{B}_{\bullet}(\mathbb{N})$. Take now
the total complex of this double complex to get the desired
resolution.

The Hochschild cohomology of a presheaf $\mathbb{A}$ is defined to
be the relative Yoneda cohomology of $\mathbb{A}$.

That is, $$\mathbf{H}^{\bullet}(\mathbb{A},-)=\mathbf{Ext}_{\mathbb{A}-\mathbb{A}}^{\bullet}(\mathbb{A},-).$$
It plays a crucial role in the study of deformations of diagrams of
algebras and it has the same rich structure as the Hochschild
cohomology of a single algebra.

If $\mathbb{P}_\bullet\rightarrow\mathbb{A}$ is a relative
projective allowable resolution of $\mathbb{A}$ then
$\mathbf{H^\bullet}(\mathbb{A},-)$ is the homology of the complex
$\mathbf{Hom}_{\mathbb{A}-\mathbb{A}}(\mathbb{P}_\bullet,-)$.

To each presheaf of algebras $\mathbb{A}$ over $\mathcal{C}$ we can
associate a single algebra $\mathbb{A}!=$ row-finite
$\mathcal{C}\times\mathcal{C}$ matrices $(a_{ij})$ with
$a_{ij}\in\mathbb{A}^i$ if $i\leq j$ and $a_{ij}=0$ otherwise. The
addition is componentwise and the multiplication
$(a_{ij})(b_{ij})=(c_{ij})$ is induced by the matrix multiplication
with the understanding that, for $h\leq i\leq j$, the summand
$a_{hi}b_{ij}$ of $c_{hj}$ is regarded as
$a_{hi}b_{ij}=a_{hi}\varphi^{hi}(b_{ij})$. For our purpose it is
convenient to use the equivalent representation
$\mathbb{A!}=\prod_{i\in\mathcal{C}}\coprod_{i\leq
j}\mathbb{A}^i\varphi^{ij}$, as $k$-bimodule. Here $\varphi^{ij}$
serve to distinguish distinct copies of $\mathbb{A}^i$ from one
another. The general element of $\mathbb{A}^i\varphi^{ij}$ will be
denoted $a^{i}\varphi^{ij}$. The multiplication is defined
componentwise and subject to the rule:
$(a^h\varphi^{hi})(a^j\varphi^{jl})=a^h\varphi^{hi}(a^j)\varphi^{hl}$
if $i=j$ and $0$ otherwise.

Let $1_i$ the unit element of $\mathbb{A}^i$. Since
$(a^h\varphi^{hi})(1_i\varphi^{ij})=a^h\varphi^{hj}$ and
$(1_i\varphi^{hi})(a^i\varphi^{ij})=\varphi^{hi}(a^i)\varphi^{hj}$
we may abbreviate $1_i\varphi^{ij} $ to $\varphi^{ij}$. The maps
$\varphi^{ij}$ are then elements of $\mathbb{A}!$ and
$\varphi^{hi}\varphi^{ij}=\varphi^{hj}$;
$\varphi^{hi}\varphi^{jl}=0$ if $i\neq j$.

We define the functor
!$\;$:$\;$$\mathbb{A}$-bimod$\longrightarrow\mathbb{A}!$-bimod, such
that $\mathbb{A}\longrightarrow\mathbb{A}!$, by setting for any
$\mathbb{A}$-bimodule $\mathbb{M}$,
$\mathbb{M}!=\prod_{i\in\mathcal{C}}\coprod_{i\leq
j}\mathbb{M}^i\varphi^{ij}$ as a $k$-bimodule. The actions of
$\mathbb{A}!$ are defined by:
$$(a^h\varphi^{hi})(m^i\varphi^{ij})=a^hT_{\mathbb{M}}^{hi}(m^i)\varphi^{hj}$$
$$(m^h\varphi^{hi})(a^i\varphi^{ij})=m^h\varphi^{hi}(a^i)\varphi^{hj}$$
$$(a^h\varphi^{hi})(m^j\varphi^{jl})=0=(m^h\varphi^{hi})(a^j\varphi^{jl}),
\mathrm{if} i\neq j.$$ For $\eta\in
Hom_{\mathbb{A}-\mathbb{A}}(\mathbb{N}, \mathbb{M})$ define
$\eta!\in Hom_{\mathbb{A}!-\mathbb{A}!}(\mathbb{N}!, \mathbb{M}!)$
by $\eta!(n^i\varphi^{ij})=\eta^i(n^i)\varphi^{ij}$.

We will use the following proposition due to M. Gerstenhaber and S.
D. Schack.
\begin{proposition}
The functor $!:\mathbb{A}$-bimod$\longrightarrow\mathbb{A}!$-bimod
is exact, preserves allowability and is full and faithful.
\end{proposition}
\begin{proof}
see[2]

\end{proof}

In fact, M. Gerstenhaber and S. D. Schack  proved in [2] the
``Special Cohomology Comparison Theorem" (SCCT).

\begin{it} $\mathbf{Theorem}$$\mathbf{(SCCT).}$
Let $\mathcal{C}$ be an arbitrary poset and $\mathbb{A}$ a presheaf
over $\mathcal{C}$. The functor ! induces an isomorphism of relative
Yoneda cohomologies
$Ext^{\bullet}_{\mathbb{A}-\mathbb{A}}((-),(-))\cong
Ext^{\bullet}_{\mathbb{A}!-\mathbb{A}!}((-)!,(-)!).$ In particular,
we have an isomorphism of relative Hochschild cohomologies
$H^{\bullet}(\mathbb{A},(-))\cong H^{\bullet}(\mathbb{A}!,(-)!).$
\end{it}

An important consequence of this theorem is that the deformation
theories  of $\mathbb{A}$ and of $\mathbb{A!}$ are equivalent, if
the poset $\mathcal{C}$ has a terminator. Another is that
$H^{\bullet}(\mathbb{A}!,\mathbb{A}!)$ has a $G$-algebra structure.
These results can be found in their full generalization to diagrams
in [2], but we will not deal with them here. We will however
generalize the SCCT to derived categories and prove theorems 3.9 and
4.1. The SCCT  follows as a corollary from these theorems. To do
this we need to introduce a subcategory of the category of
$\mathbb{A!}$-bimod.
The image of ! lies in a full subcategory of $\mathbb{A}!$-bimod.
This is the category of $\mathbf{aligned}$ bimodules,
$\mathbb{A!}$-albimod. The main reason to consider it here is that
the functor ! has a left adjoint when restricted to
$!:\mathbb{A}$-bimod $\longrightarrow\mathbb{A!}$-albimod. Thus, for
every $\mathbb{A!}$ bimodule $X$ we set $X_{al}=\prod_{i\in
\mathcal{C}}\coprod_{i\leq j}\varphi^{ii}X\varphi^{jj}$ with the
obvious $\mathbb{A!}$ bimodule structure.

\begin{definition}
An $\mathbb{A!}$ bimodule $X$ is said to be $\mathbf{aligned}$ if
the $k$ linear map $X\longrightarrow \prod_{i\in
\mathcal{C}}\prod_{j\in \mathcal{C}}\varphi^{ii}X\varphi^{jj}$,
$x\longrightarrow<\varphi^{ii}x\varphi^{jj}>$ induces an
$\mathbb{A!}$ bimodule isomorphism $\alpha_{X} : X\longrightarrow
X_{al}=\prod_{i\in \mathcal{C}}\coprod_{i\leq
j}\varphi^{ii}X\varphi^{jj}.$
\end{definition}
For each $\mathbb{A}!$-bimodule map $f:X\longrightarrow Y$, the
restriction of $f$ to $\varphi^{ii}X\varphi^{jj}$ is a $k$ linear,
even a $\mathbb{A}^i$-$\mathbb{A}^j$-bimodule map
$f^{ij}:\varphi^{ii}X\varphi^{jj}\longrightarrow
\varphi^{ii}Y\varphi^{jj}$ since $f(\varphi^{ii}x\varphi^{jj})=
\varphi^{ii}f(x)\varphi^{jj}$ lies in $\varphi^{ii}Y\varphi^{jj}$.
Thus, $f$ gives rise to a family of $k$ linear maps
${f^{ij}:\varphi^{ii}X\varphi^{jj}\longrightarrow\varphi^{ii}Y\varphi^{jj}}$
such that $f^{hj}(a^h\varphi^{hi}\cdot x)=a^h\varphi^{hi}\cdot
f^{ij}(x)$ and $f^{iq}(x\cdot a^j\varphi^{jq})=f^{ij}(x)\cdot
a^j\varphi^{jq}$ $\forall x\in\varphi^{ii}X\varphi^{jj}$,
$a^h\in\mathbb{A}^h$, $a^j\in\mathbb{A}^j$ and $h\leq i\leq j\leq q$
in $\mathcal{C}$.

In fact these are exactly the conditions necessary on such a
collection of maps for $f_{al}=\prod_{i\in\mathcal{C}}\coprod_{i\leq
j}f^{ij} $ to be an $\mathbb{A}!$-bimodule map
$X_{al}\longrightarrow Y_{al}$. One can easily see that
$\mathbb{A}!$-albimod is abelian, and that both the inclusion
functor $inc:\mathbb{A}!$-albimod$\longrightarrow\mathbb{A}!$-bimod
and the alignment functor
$(-)_{al}:\mathbb{A}!$-bimod$\longrightarrow\mathbb{A}!$-albimod,
$X\longrightarrow X_{al}$ are exact and preserve allowability and
that
 $\alpha : Id_{\mathbb{A!}-albimod}\longrightarrow
(-)_{al}\circ inc $ is a natural isomorphism.

Now, we describe a method of producing relative projective allowable
resolutions of aligned bimodules of the form $\mathbb{N}!$ that we
will use to replace complexes of aligned bimodules with relative
projective ones in a suitable derived category. We begin with a
result due to M. Gerstenhaber and S. D. Schack.
\begin{proposition}
1. For each $i\leq j$ in $\mathcal{C}$ the restriction functor
$(-)^{ij}:\mathbb{A}!$-albimod$\longrightarrow\mathbb{A}^i$-mod-$\mathbb{A}^j$,
$X\longrightarrow\varphi^{ii}X\varphi^{jj}$ is exact and preserves
allowability.\\
2. The functor $(-)^{ij}$ has a left adjoint $L_{ij}$ that preserves
relative projectivity.
\end{proposition}
\begin{proof}
Part 1 is obvious. For 2, define
$L_{ij}:\mathbb{A}^i$-mod-$\mathbb{A}^j\longrightarrow\mathbb{A}!$-albimod
as follows: $L_{ij}(N)^{hl}=\begin{cases}
\mathbb{A}^h\otimes_{\mathbb{A}^i}|N|_{jl}$\;$ $\;$$\;$ $\;$
\mathrm{if}$\;$ $\;$ h\leq i\leq j\leq
l\\
 $0$$\;$ $\;$$\;$ $\;$$\;$ $\;$ \mathrm{otherwise}\\
 \end{cases}$

Here, $|N|_{jl}$ is $N$ viewed as a left $\mathbb{A}^i$-module and a
right $\mathbb{A}^l$-module via the map $\varphi^{jl}$. The actions
of $\mathbb{A}!$ are given by
$$a^r\varphi^{rh}(a^h\otimes n)=a^r\varphi^{rh}(a^h)\otimes n\in
L_{ij}(N)^{rl}$$ $$(a^h\otimes n)a^l\varphi^{lm}=a^h\otimes
n\varphi^{jl}(a^l)\in L_{ij}(N)^{hm},$$ for $a^h\otimes n\in
L_{ij}(N)^{hl}$ and $a^r\varphi^{rh},
a^l\varphi^{lm}\in\mathbb{A}!$.

One can check now that we have a natural isomorphism
$$Hom_{\mathbb{A}!-\mathrm{albimod}}(L_{ij}(N),X)\leftrightarrows
Hom_{\mathbb{A}^{i}-\mathbb{A}^j}(N,X^{ij})$$ for all
$X\in\mathbb{A}!$-albimod and
$N\in\mathbb{A}^{i}$-mod-$\mathbb{A}^j$.

If $P\in\mathbb{A}^i$-mod-$\mathbb{A}^j$ is relative projective then
the natural isomorphism\\
$Hom_{\mathbb{A}!-{\mathrm{albimod}}}(L_{ij}(P),-)\cong
Hom_{\mathbb{A}^{i}\mathbb{A}^{j}}(P,(-)^{ij})=Hom_{\mathbb{A}^{i}\mathbb{A}^{j}}(P,-)
\circ(-)^{ij}$ is a composite of functors which preserve allowable
epimorphisms, so $L_{ij}(P)$ is relative projective. ( for more
details see [2])
\end{proof}

Modeled on the M. Gerstenhaber - S. D. Schack resolution
$\mathcal{S}_\bullet$, C. B. Kullmann obtained in [3] an allowable
resolution
$\mathcal{T}_{\bullet}\mathbb{N}\longrightarrow\mathbb{N}!$ in
$\mathbb{A}!$-albimod as follows. For $p\geq 0$ let
$\mathcal{T}_{p}\mathbb{N}=\coprod_{\sigma\in\mathcal{C}^{[p]}}
\mathcal{T}_{p}^{\sigma}\mathbb{N}$, where the coproduct is taken in
$\mathbb{A}!$-albimod (  constructed by applying $(-)_{al}$ to that
in $\mathbb{A}$-bimod ), where
$\mathcal{T}_{p}^{\sigma}\mathbb{N}=L_{d\sigma,c\sigma}
(\mathbb{A}^{d\sigma}\otimes_{\mathbb{A}^{c\sigma}}\mathbb{N}^{c\sigma})$.

For $h\leq d\sigma\leq c\sigma\leq l$ we have a natural isomorphism
$(\mathcal{T}_p^{\sigma}\mathbb{N})^{hl}=\mathbb{A}^{h}\otimes_{\mathbb{A}^{d\sigma}}
|\mathbb{A}^{d\sigma}\otimes_{\mathbb{A}^{c\sigma}}\mathbb{N}^{c\sigma}|_{c\sigma,l}
\cong\mathbb{A}^{h}\otimes_{\mathbb{A}^{c\sigma}}|\mathbb{N}^{c\sigma}|_{c\sigma.l}$
 and we use this identification to define the differentials. If
$p\geq 1$ we define
$d:\mathcal{T}_{p}\mathbb{N}\longrightarrow\mathcal{T}_{p-1}\mathbb{N}$
as a sum $d=\sum_{r=0}^p(-1)^rd_r$, where each $d_r$ is determined
by its restriction to $\mathcal{T}_{p}^{\sigma}\mathbb{N}$ and for
$h\leq d\sigma\leq c\sigma\leq l$ and $a\otimes
n\in(\mathcal{T}_p^{\sigma}\mathbb{N})^{hl}=\mathbb{A}^{h}\otimes_{\mathbb{A}^{c\sigma}}
|N^{c\sigma}|_{c\sigma,l}$, we have $d_r^\sigma(a\otimes n)=a\otimes
T_{\mathbb{N}}^{c\sigma_r,c\sigma}(n)\in\mathbb{A}^{h}\otimes_{\mathbb{A}^{c\sigma_r}}
|\mathbb{N}^{c\sigma_r}|_{c\sigma_r,l}=(\mathcal{T}_{p-1}^{\sigma}\mathbb{N})^{hl}$.
If $p=0$ the map $\varepsilon^{\mathcal{T}}:
\mathcal{T}_0\mathbb{N}=\coprod_{i}\mathcal{T}_0^{(i)}\mathbb{N}\longrightarrow\mathbb{N}!$
is determined by
$(\mathcal{T}_0^{(i)}\mathbb{N})^{hl}=\mathbb{A}^h\otimes_{\mathbb{A}^i}|\mathbb{N}^i|_{il}
\longrightarrow\mathbb{N}!^{hl}=\mathbb{N}^h\varphi^{hl}$, $a\otimes
n\longrightarrow aT_{\mathbb{N}}^{hi}(n)\varphi^{hl}$, for $h\leq
i\leq l$.

It is easy to check that
$\mathcal{T}_{\bullet}\mathbb{N}\longrightarrow\mathbb{N}!$ is a
chain complex and it is in fact an allowable resolution since it has
a contracting homotopy induced by
$\kappa_p:(\mathcal{T}_p^{\sigma}\mathbb{N})^{hl}\longrightarrow(\mathcal{T}_{p+1}^{(h,\sigma)}
\mathbb{N})^{hl}$, $\kappa_p=identity$, where $(h,\sigma)$ is the
simplex $(h,\sigma(0),\ldots,\sigma(p))$ if $ h\leq \sigma(0)$ and
$\mathcal{T}_{p+1}^{(h,\sigma)}\mathbb{N}=0$ if $h\nleq d\sigma$.

In general $\mathcal{T}_p\mathbb{N}$ is not a relative projective
aligned $\mathbb{A}!$-bimodule, but it is when each $\mathbb{N}^i$
is relative projective $\mathbb{A}^i$-bimodule. To obtain  a
relative projective aligned resolution, for each
$\mathbb{A}$-bimodule $\mathbb{N}!$, take the relative projective
resolution $\mathcal{B}_{\bullet}(\mathbb{N})$, apply ! and then
$\mathcal{T}_{\bullet}$ to obtain a double complex. Now, take the
total complex to obtain the desired resolution.

We conclude this section with a result which connects
$\mathcal{T}_{\bullet}$ and $\mathcal{S}_{\bullet}$ via a left
adjoint  of !. Because the only source for the following theorem is
[3]  the proof is included in the Appendix A.
\begin{theorem}
1. The functor
$!:\mathbb{A}$-bimod$\longrightarrow\mathbb{A}!$-albimod admits a
left adjoint
$\raisebox{.45ex}{\textup{\textbf{!`}}}:\mathbb{A}!$-albimod$\longrightarrow\mathbb{A}$-bimod.\\
2. There are natural isomorphisms
$\mathcal{T}_p\mathbb{N}\raisebox{.45ex}{\textup{\textbf{!`}}}\longrightarrow\mathcal{S}_p\mathbb{N}$
which induce a natural isomorphism of complexes
$(\mathcal{T}_{\bullet}\mathbb{N}\longrightarrow\mathbb{N}!)\raisebox{.45ex}{\textup{\textbf{!`}}}$
 and
$(\mathcal{S}_{\bullet}\mathbb{N}\longrightarrow\mathbb{N})$.
\end{theorem}
\begin{proof}
see  Appendix A.

\end{proof}

 %%%%%%%%%the third secton%%%%%%%%%%%%%%%%%%
\section{Derived categories and Hochschild cohomology }
Let $Kom^{-}(\mathbb{A}-\mathrm{bimod})$ the category of bounded to
the right complexes of $\mathbb{A}$-bimodules
$$\mathbb{M_\bullet}:=\xymatrix{\cdots
\mathbb{M}_n\ar[r]&\cdots&%\mathbb{M}_2
\cdots\ar[r]&\mathbb{M}_1\ar[r]& \mathbb{M}_0\ar[r]&0}$$

A map between two complexes $\mathbb{M_\bullet}$ and
$\mathbb{N_\bullet}$ is a collection of maps
$f=(f_i):\mathbb{M}_i\rightarrow\mathbb{N}_i$, one for each positive
integer $i$, which commute with the differentials of
$\mathbb{M_\bullet}$ and $\mathbb{N_\bullet}$. We do not require the
maps defining the complexes or the maps between complexes to be
$k$-split.

Similarly, we define $Kom^{-}(\mathbb{A}!-\mathrm{albimod})$ and
$Kom^{-}(\mathbb{A}!-\mathrm{bimod})$.

\begin{definition}
1) A map $f:\xymatrix{M_{\bullet}\ar[r]&N_{\bullet}}$ in
$Kom^{-}(\mathbb{A}!-\mathrm{bimod})$ (or
$Kom^{-}(\mathbb{A}!-\mathrm{albimod})$) is a $\mathbf{relative}$
$\mathbf{quasi}$-$\mathbf{isomorphism}$ if its cone
$\mathcal{C}(f)_\bullet$ is contractible when considered as a
complex of $k$-bimodules.\\2) A map
$f:\xymatrix{\mathbb{M_{\bullet}}\ar[r]&\mathbb{N}_{\bullet}}$ in
$Kom^{-}(\mathbb{A}-\mathrm{bimod})$ is a $\mathbf{relative}$
$\mathbf{quasi}$-$\mathbf{isomorphism}$ if the maps of complexes
$f^i:\xymatrix{\mathbb{M}_{\bullet}^i\ar[r]&\mathbb{N}_{\bullet}^i}$
have  contractible cones, when considered as complexes of
$k$-bimodules, for all $i\in\mathcal{C}$.
\end{definition}
The word ``$\mathbf{relative}$" in the above definition is used as a
reminder to the reader that Yoneda and Hochschild cohomologies are
relative theories,  since $k$ is a commutative ring that is not
necessarily a field. It is the relative Yoneda  groups that we want
to view as homomorphism groups in a suitable category.
\begin{proposition}
Let $A$ be any $k$-algebra and $f:M_{\bullet}\longrightarrow
N_{\bullet}$ a map of complexes of $A$ bimodules in
$Kom^{-}(A-bimod)$. Then, $f$ is a relative quasi-isomorphism if and
only if there exists $\gamma:N_{\bullet}\longrightarrow M_{\bullet}$
a map of complexes of $k$-bimodules such that $f\gamma\sim
id_{N_{\bullet}}$  and $\gamma f\sim id_{M_{\bullet}}$ in
$Kom^{-}(k-bimod)$, where `$\sim$' stands for homotopy equivalence.
\end{proposition}

\begin{proof}
$'\Rightarrow'$

Assume that $f$ is a relative quasi-isomorphism. Thus
$\mathcal{C}{(f)}_{\bullet }$ is contractible when regarded as a
complex of $k$-bimodules, so there exist
$s=(s_n):\mathcal{C}{(f)}_{\bullet}^{n-1}\longrightarrow\mathcal{C}{(f)}_{\bullet}^{n}$
maps of $k$-bimodules such that
$sd_{\mathcal{C}{(f)}_{\bullet}}+d_{\mathcal{C}{(f)}_{\bullet}}s=id$.
We may assume that
\begin{center}$s=\left(%
\begin{array}{cc}
  \alpha & \gamma\\
  \beta& \delta \\
\end{array}%
\right)\mathrm{and}$\end{center} \begin{center}$
 d_{\mathcal{C}{(f)}_{\bullet}}=\left(%
\begin{array}{cc}
  -d_{M_\bullet} & 0\\
  f & d_{N_\bullet} \\
\end{array}%
\right),$\end{center} where $\alpha:M_{\bullet -1}\longrightarrow M_{\bullet}$,
$\beta:M_{\bullet -1}\longrightarrow N_{\bullet +1}$,
$\gamma:N_{\bullet}\longrightarrow M _{\bullet}$ and
$\delta:N_{\bullet}\longrightarrow N_{\bullet +1}$ are $k$ linear
maps. Since
$sd_{\mathcal{C}{(f)}_{\bullet}}+d_{\mathcal{C}{(f)}_{\bullet}}s=id$,
we obtain  $-\alpha d_{M_{\bullet}}+\gamma
f-d_{M_{\bullet}}\alpha=id_{M_{\bullet}}$, $-\beta
d_{M_{\bullet}}+\delta f+f\alpha+d_{N_{\bullet}}\beta=0$, $\gamma
d_{N_\bullet}-d_{M_\bullet}\gamma=0 $ and $\delta
d_{N_\bullet}+f\gamma+d_{N_\bullet}\delta=id_{N_\bullet}$.

Thus, $\gamma$ is a map of complexes of $k$-bimodules and since
$\delta
d_{N_{\bullet}}+d_{N_{\bullet}}\delta=id_{N_{\bullet}}-f\gamma$ and
$\alpha d_{M_{\bullet}}+d_{M_{\bullet}}\alpha=\gamma
f-id_{M_{\bullet}}$, we have $f\gamma \sim id_{N_{\bullet}}$ and
$\gamma f\sim id_{M_{\bullet}}$ in $Kom^{-}{(k-bimod)}$.

$'\Leftarrow'$

Assume $f\gamma \sim id_{N_{\bullet}}$ and $\gamma f\sim
id_{M_{\bullet}}$ in $Kom^{-}{(k-bimod)}$, so there are maps
$s^{N_{\bullet}}$ and $s^{M_{\bullet}}$ such that $f\gamma -
id_{N_{\bullet}}=s^{N_{\bullet}}d_{N_{\bullet}}+d_{N_{\bullet}}s^{N_{\bullet}}$
and $\gamma f-
id_{M_{\bullet}}=s^{M_{\bullet}}d_{M_{\bullet}}+d_{M_{\bullet}}s^{M_{\bullet}}$.

The map $s^{\mathcal{C}{(f)}}_{\bullet}=\left(%
\begin{array}{cc}
  s^{M_{\bullet}}+\gamma(s^{N_{\bullet}}f-fs^{M_{\bullet}}) & \gamma \\
  s^{N_{\bullet}}(fs^{M_{\bullet}}-s^{N_{\bullet}}f) & -s^{N_{\bullet}} \\
\end{array}%
\right) $ is a homotopy. Indeed,
$$s^{\mathcal{C}{(f)}}_{\bullet}d_{\mathcal{C}{(f)}_{\bullet}}+
d_{\mathcal{C}{(f)}_{\bullet}}s^{\mathcal{C}{(f)}}_{\bullet}=$$
$$=\left(%
\begin{array}{cc}
  id_{M_{\bullet}}-\gamma s^{N_{\bullet}}fd_{M_{\bullet}}+\gamma fs^{M_{\bullet}}d_{M_{\bullet}}-d_{M_{\bullet}}\gamma s^{N_{\bullet}}f+d_{M_{\bullet}}\gamma fs^{M_{\bullet}}& 0 \\
  s^{N_{\bullet}}s^{N_{\bullet}}fd_{M_{\bullet}}+f\gamma s^{N_{\bullet}}f- d_{N_{\bullet}}s^{N_{\bullet}}s^{N_{\bullet}}f-s^{N_{\bullet}}f\gamma f& id_{N_{\bullet}} \\
\end{array}%
\right) $$

$\;$

= $\left(%
\begin{array}{cc}
   id_{M_{\bullet}} & 0 \\
  0 &  id_{N_{\bullet}} \\
\end{array}%
\right)=$ $id_{\mathcal{C}{(f)}_{\bullet}}$. Thus
$\mathcal{C}{(f)}_{\bullet}$ is contractible in $Kom^{-}(k-bimod).$
\end{proof}

Proposition 3.2. allows us to conclude that if any two of $f, g$ or
$fg$ are relative quasi-isomorphisms then so is the third. We prove
now the following

\begin{proposition} The class of relative quasi-isomorphisms in the homotopic
category $\mathcal{K}^{-}(\mathbb{A}-\mathrm{bimod})$ is localizing.

\end{proposition}

\begin{proof}

We showed already that the class of relative quasi-isomorphisms is
closed under the composition of maps. To conclude this class is
localizing we need to justify two facts:

1) The extension conditions: For every $ f\in
Mor_{\mathcal{K}^{-}(\mathbb{A}-\mathrm{bimod})}$ and $s$ relative
quasi-isomorphism there exist
$g\in$$Mor_{\mathcal{K}^{-}(\mathbb{A}-\mathrm{bimod})}$ and $t$
relative quasi-isomorphism such that the following square

$$\xymatrix{\mathbb{N}_{\bullet}\ar[r]^f
\ar[d]_{t}&\mathbb{M}_{\bullet}\ar[d]^
{s}\\
\mathbb{K}_{\bullet}\ar[r]^{g}& \mathbb{L}_{\bullet}}$$ (resp.
$$\xymatrix{\mathbb{L}_{\bullet}\ar[r]^{g}\ar[d]_{s}&\mathbb{K}_{\bullet}\ar[d]^{t}\\
\mathbb{M}_{\bullet}\ar[r]^{f}& \mathbb{N}_{\bullet}}$$

is commutative.

2) Given  $f,g$ two morphisms from $\mathbb{N}_\bullet$ to
$\mathbb{M}_\bullet$, the existence of $s$ relative
quasi-isomorphism with $sf=sg$ is equivalent to the existence of $t$
relative quasi-isomorphism with $ft=gt$.

The proof of  theorem 4, chapter 3 in [5],  which states that the
class of quasi-isomorphisms (not relative) in the homotopic category
of an abelian category is localizing, can be used entirely so we
will not reproduce it here.  One needs to note for 1) that the cone
of the map $t$ constructed there is the same, in
$\mathcal{K}^{-}(\mathbb{A}-\mathrm{bimod})$, as the cone of $s$;
and for 2) that the cone of the map $t$ constructed there is the
cone of $s$ shifted by 1. Thus in both cases $t$ is a relative
quasi-isomorphism.

\end{proof}

Remark that the same result is true for
$\mathcal{K}^{-}(\mathbb{A}!-\mathrm{bimod})$ and
$\mathcal{K}^{-}(\mathbb{A}!-\mathrm{albimod})$.

We now define the relative derived categories by formally inverting
all relative quasi-isomorphisms.
\begin{definition}
Let $\mathcal{A}$ be any of the categories $\mathbb{A}$-bimod,
$\mathbb{A}!$-bimod or $\mathbb{A}!$-albimod and $\sum$ the
appropriate class of relative quasi-isomorphisms.
$$\mathcal{D}_{k}^{-}(\mathcal{A}):=\mathcal{K}^{-}(\mathcal{A})(\Sigma^{-1}),$$
where $\mathcal{K}^{-}$ is the corresponding homotopy category.
\end{definition}
Because $\sum$ is localizing  we may regard the morphisms, in any of
the relative derived categories defined above, as equivalence
classes of diagrams $$\xymatrix{&U\ar[dl]_t\ar[dr]^g\\
                    X&&Y}$$
The maps $t$ and $g$ are morphisms in the homotopy category with
$t\in\sum$. These diagrams are usually called roofs and we adopt
this terminology. In addition, because $\sum$ is a localizing class
the relative derived categories defined above are triangulated.

We begin studying the objects of $\mathcal{D}_k^{-}
({\mathbb{A}}-{\mathrm{bimod}})$  with the complexes of relative
projective bimodules.\newpage
\begin{lemma}
Let $\mathbb{P}_{\bullet}$ be a complex of relative projective
$\mathbb{A}$-$\mathrm{bimodules}$ and
$\xymatrix{\mathbb{R}_{\bullet}\ar[r]^f&\mathbb{P}_{\bullet}}$ a
relative quasi-isomorphism. We have
\begin{center}$Mor_{\mathcal{K}^{-}(\mathbb{A}-\mathrm{bimod})}
(\mathbb{P}_{\bullet},\mathcal{C}{(f)}_{\bullet})=0.$\end{center}
\end{lemma}
\begin{proof}
Because $f$ is a relative quasi-isomorphism the cone
$\mathcal{C}{(f)}_{i}$ is acyclic and allowable $(\forall)
i\in\mathcal{C}$. Given $g\in
Mor_{\mathcal{K}^{-}(\mathbb{A}-\mathrm{bimod})}
(\mathbb{P}_{\bullet},\mathcal{C}{(f)}_{\bullet})$  we show that
$g=(g)_i:\mathbb{P}_i\longrightarrow\mathcal{C}(f)_i, i\geq 0$ is
homotopic to 0 inductively. Since $\mathbb{P}_0$ is a complex of
relative projective $\mathbb{A}$-bimodules we obtain that the map
$g_0$ from $\mathbb{P}_0$ to $\mathcal{C}{(f)}_0$ can be lifted to a
map $\delta_0:\mathbb{P}_0\longrightarrow \mathcal{C}(f)_1$ such
that $d_{\mathcal{C}(f)_1}\delta_0=g_0$. The image of
$g_1-\delta_0d_{\mathbb{P}_1}$ is contained in the image of
$d_{\mathcal{C}(f)_1}$ so it has a lifting
$\delta_1:\mathbb{P}_1\longrightarrow\mathcal{C}(f)_2$ such that
$d_{\mathcal{C}(f)_2}\delta_1=g_1-\delta_0d_{\mathbb{P}_1}$. Now,
the image of $g_2-\delta_1d_{\mathbb{P}_2}$ is contained in the
image of  $d_{\mathcal{C}(f)_2}$ and the conclusion follows
inductively.
\end{proof}
\begin{proposition}
Let $\mathbb{P}_{\bullet}$ be a complex of relative projective
bimodules in $Kom^{-}(\mathbb{A}-\mathrm{bimod})$. The canonical map
\begin{center}$\xymatrix{
Mor_{\mathcal{K}^{-}(\mathbb{A}-\mathrm{bimod})}(\mathbb{P}_{\bullet},\mathbb{M}_{\bullet})
\ar[r]^{can}&Mor_{\mathcal{D}^{-}_{k}(\mathbb{A}\mathrm{-bimod})}
(\mathbb{P}_{\bullet},\mathbb{M}_{\bullet})}$\end{center} is an isomorphism for
all $\mathbb{M}_{\bullet}\in\ Kom^{-}(\mathbb{A}-\mathrm{bimod})$.
\end{proposition}\newpage
\begin{proof}
To prove the injectivity let
$\xymatrix{\mathbb{P}_{\bullet}\ar[r]^{\alpha}&\mathbb{M}_{\bullet}}$
and
$\xymatrix{\mathbb{P}_{\bullet}\ar[r]^{\beta}&\mathbb{M}_{\bullet}}$
such that their corresponding roofs:
$$\xymatrix{&\mathbb{P}_{\bullet}\ar[dl]_{id}\ar[dr]^{\alpha}\\
\mathbb{P}_{\bullet}&&\mathbb{M}_{\bullet}}\mathrm{and}
\xymatrix{&\mathbb{P}_{\bullet}\ar[dl]_{id}\ar[dr]^{\beta}\\
\mathbb{P}_{\bullet}&&\mathbb{M}_{\bullet}}$$\\
are equivalent in $\mathcal{D}^{-}_k(\mathbb{A}-\mathrm{bimod})$.
Thus, we have the commutative diagram in
$\mathcal{K}^{-}(\mathbb{A}-\mathrm{bimod})$
 $$\xymatrix{&&\mathbb{X}_{\bullet}\ar[dl]_{a}\ar[dr]^{b}\\
 &\mathbb{P}_{\bullet}\ar[dl]_{id}\ar[drrr]^{\alpha}&&\mathbb{P}_{\bullet}\ar[dlll]_{id}\ar[dr]^{\beta}\\
 \mathbb{P}_{\bullet}&&&&\mathbb{M}_{\bullet}}$$
We obtain $a=b$ and $\alpha a=\beta b$.

To check that $\alpha=\beta$, apply
$Mor_{\mathcal{K}^{-}(\mathbb{A}-bimod)}(\mathbb{P_{\bullet}},-)$ to
the distinguished triangle
$\xymatrix{\mathbb{X}_{\bullet}\ar[r]^{a}&
\mathbb{P}_{\bullet}\ar[r]&\mathcal{C}(a)_{\bullet}\ar[r]&\mathbb{X}_{\bullet}[1]}$
and use previous lemma  to see that
$Mor_{\mathcal{K}^{-}(\mathbb{A}-bimod)}(\mathbb{P_{\bullet}},\mathcal{C}(a)_{\bullet})$
$=0$. This implies the existence of a map $c$ such that
$ac=id_{\mathbb{P}_\bullet}$ in
$\mathcal{K}^{-}(\mathbb{A}-\mathrm{bimod})$ and the injectivity
follows from here.

For a morphism in $Mor_{\mathcal{D}_k^{-}
({\mathbb{A}}-{\mathrm{bimod}})}(\mathbb{P}_{\bullet},
\mathbb{M}_{\bullet})$ represented by the roof
$$
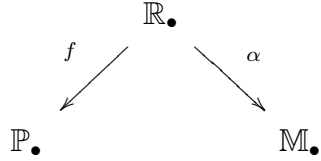
\xymatrix{&\mathbb{R}_{\bullet}\ar[dr]^{\alpha}\ar[dl]_{f}\\
\mathbb{P}_{\bullet}&&\mathbb{M}_{\bullet}}$$ the distinguished
triangle
$\xymatrix{\mathbb{R}_{\bullet}\ar[r]^f&\mathbb{P}_{\bullet}
\ar[r]&\mathcal{C}{(f)}_{\bullet}\ar[r]&\mathbb{R}_{\bullet}{[1]}}$
induces a long exact sequence by applying
$Mor_{\mathcal{K}^{-}(\mathbb{A}-\mathrm{bimod})}(\mathbb{P}_{\bullet},(-))$
to it. Again, by the previous lemma
 $Mor_{\mathcal{K}^{-}(\mathbb{A}-\mathrm{bimod})}
(\mathbb{P}_{\bullet},\mathcal{C}{(f)}_{\bullet})=0$, thus the map
$\xymatrix{Mor_{\mathcal{K}^{-}(\mathbb{A}-\mathrm{bimod})}
(\mathbb{P}_{\bullet},\mathbb{R}_{\bullet})\ar[r]^f&
Mor_{\mathcal{K}^{-}(\mathbb{A}-\mathrm{bimod})}
(\mathbb{P}_{\bullet},\mathbb{P}_{\bullet})}$ is onto, so
$(\exists)$ a map
$\xymatrix{\mathbb{P}_{\bullet}\ar[r]^s&\mathbb{R}_{\bullet}}$ such
that $fs=id_{\mathbb{P}_{\bullet}}$ in
$\mathcal{K}^{-}(\mathbb{A}-\mathrm{bimod})$. Since $f$ is a
relative quasi-isomorphism $s$ is a relative quasi-isomorphism, so
we have the commutative diagram:
$$
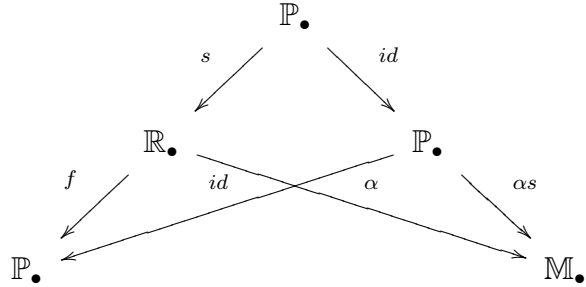
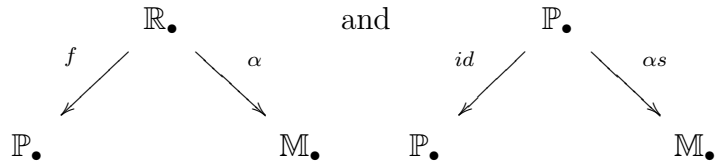
\xymatrix{&&\mathbb{P}_{\bullet}\ar[dl]_s\ar[dr]^{id}\\
                   &\mathbb{R}_{\bullet}\ar[dl]_f\ar[drrr]^{\alpha}&&\mathbb{P}_{\bullet}
                   \ar[dlll]_{id}\ar[dr]^{\alpha s}\\
                   \mathbb{P}_{\bullet}&&&&\mathbb{M}_{\bullet}}$$
Thus, the roofs
$$\xymatrix{&\mathbb{R}_{\bullet}\ar[dl]_f\ar[dr]^{\alpha}\\
            \mathbb{P}_{\bullet}&&\mathbb{M}_{\bullet}}\mathrm{and}
            \xymatrix{&\mathbb{P}_{\bullet}\ar[dl]_{id}\ar[dr]^{\alpha
            s}\\
            \mathbb{P}_{\bullet}&&\mathbb{M}_{\bullet}}$$
are equivalent and since the second is the image of $\alpha s$ the
surjectivity is proved.
\end{proof}
Note that relative projective complexes in
$Kom^{-}(\mathbb{A}!-\mathrm{bimod})$ and
$Kom^{-}(\mathbb{A}!-\mathrm{albimod})$ satisfy the same property.

We prove now that each complex of $\mathbb{A}$-bimodules is relative
quasi-isomorphic to a complex of relative projective bimodules. For
this we need the following
\begin{proposition}
Let $A$ be a $k$ algebra and assume that we have a double complex of
$A$ bimodules
$$\xymatrix{%\cdots\ar[r]^{d_3}%
&\vdots\ar[d]^{d^1}&\vdots\ar[d]^{d^0}&\vdots\ar[d]^{d_M}\\
\cdots\ar[r]^{d_2}&X_{12}\ar[r]^{d_2}\ar[d]^{d^1}&X_{02}\ar[r]^{\varepsilon_2}\ar[d]^{d^0}&M_2\ar[d]^{d_M}\ar[r]&0\\
\cdots\ar[r]^{d_1}&X_{11}\ar[r]^{d_1}\ar[d]^{d^1}&X_{01}\ar[d]^{d^0}\ar[r]^{\varepsilon_1}&M_1\ar[d]^{d_M}\ar[r]&0\\
\cdots\ar[r]^{d_0}&X_{10}\ar[r]^{d_0}&X_{00}\ar[r]^{\varepsilon_0}&M_0\ar[r]&0}$$
such that:

a) Each row is $k$ contractible. ( $i.e.$ There exist $k$-bimodule
maps

$\xymatrix{X_{(k-1)i}\ar[r]^{t_i^k}&X_{ki}}$  such that
$d_it_i^{k+1}+t_i^kd_i=id_{X_{ki}}$.)

b) The following diagrams are commutative:

$$\xymatrix{X_{ki}\ar[d]_{d^k}&X_{(k-1)i}\ar[d]^{d^{k-1}}\ar[l]_{t_i^k}\\
           X_{k(i-1)}&X_{(k-1)(i-1)}\ar[l]_{t_{i-1}^k}}
\xymatrix{X_{0i}\ar[d]_{d^0}&M_i\ar[d]^{d_M}\ar[l]_{t_i^0}\\
           X_{0(i-1)}&M_{i-1}\ar[l]_{t_{i-1}^0}}$$
           for all
           $k,i\geq 0$, Then

1.$\xymatrix{M_{\bullet}\ar[rr]^{t_{\bullet}^0}&&(TotX_{\bullet
\bullet})}$ and $\xymatrix{(TotX_{\bullet
\bullet})\ar[rr]^{\varepsilon_{\bullet}}&&M_{\bullet}}$ are maps of
complexes of $k$-bimodules, where $\varepsilon_i=0$ on $X_{jk},
j+k=i$ if $j>0$.

2. $\varepsilon_{\bullet}t_{\bullet}^0=id_{M_{\bullet}}$ and
$t_{\bullet}^0\varepsilon_{\bullet}\sim id_{TotX_{\bullet \bullet}}$
in $Kom^{-}(k-bimod),$ where $\sim $=homotopy equivalence.

\end{proposition}
\begin{proof}
1. The map $t_{\bullet}^0 $ is a map of complexes by b) and
$\varepsilon_{\bullet}$ is a map of complexes because
$d_M\varepsilon_{i+1}=d^0\varepsilon_i$ and $\varepsilon_id_i=0$.

2. The only thing to prove here is
$t_{\bullet}^0\varepsilon_{\bullet}\sim id_{TotX_{\bullet \bullet}}$
in $Kom^{-}(k-bimod).$ For $n\geq 0$ we define the map
$\xymatrix{(TotX_{\bullet \bullet})^n\ar[r]^{h^n}&(TotX_{\bullet
\bullet})^{n+1}}$ by $h^n:=(t_0^{n+1},t_1^{n},\ldots,t_n^1,0)$. It
is a simple exercise to check that $h^{\bullet}d_{TotX_{\bullet
\bullet}}+d_{TotX_{\bullet
\bullet}}h^{\bullet}=id-t_{\bullet}^0\varepsilon_{\bullet}$.
\end{proof}
\begin{theorem}

For each
$\mathbb{M}_{\bullet}\in\mathcal{D}^{-}_{k}(\mathbb{A}-\mathrm{bimod})$
there exist
$\mathcal{U}\mathbb{M}_{\bullet}\in\mathcal{D}^{-}_{k}(\mathbb{A}-\mathrm{bimod})$
and
$\xymatrix{\mathcal{U}\mathbb{M}_{\bullet}\ar[r]^{\varepsilon}&\mathbb{M}_{\bullet}}$
a relative quasi-isomorphism such that
$\mathcal{U}\mathbb{M}_{\bullet}$ is a complex of relative
projective $\mathbb{A}$-bimodules.
\end{theorem}
\begin{proof}
We described in section 2 a method of constructing a relative
projective allowable resolution
$Tot\mathcal{S}_{\bullet}\mathcal{B}_{\bullet}(\mathbb{M})\longrightarrow\mathbb{M}$,
for each $\mathbb{M}\in\mathbb{A}$-bimod. We use this for each term
$\mathbb{M}_i$ of the complex $\mathbb{M}_\bullet$, $i\geq 0$. We
obtain a double complex with augmented column
$\mathbb{M}_{\bullet}$. In addition, each row is contractible and
for all $p\in\mathcal{C}$ we obtain a double complex of
$\mathbb{A}^{p}$-bimodules which satisfies the conditions of the
previous proposition.  Thus, by taking the total complex of the
double complex with augmented column $\mathbb{M}_\bullet$ we obtain
the desired complex of relative projective $\mathbb{A}$-bimodules,
$\mathcal{U}\mathbb{M}_{\bullet}$, together with a relative
quasi-isomorphism
$\xymatrix{\mathcal{U}\mathbb{M}_{\bullet}\ar[r]^{\varepsilon}&\mathbb{M}_{\bullet}}.$
\end{proof}
Note that the same argument shows that for each complex
$\mathbb{M}_{\bullet}!\in\mathcal{D}^{-}_{k}(\mathbb{A}!-\mathrm{bimod})$
the total complex, $Tot\mathcal{T}_{\bullet}\mathbb{M}_{\bullet}$,
of the double complex $\mathcal{T}_{\bullet}\mathbb{M}_{\bullet}$
obtained by taking the allowable resolution of each $\mathbb{M}_i!$
described in section 2, gives a relative quasi-isomorphism
$\xymatrix@1{(Tot\mathcal{T}_{\bullet}\mathbb{M}_{\bullet})
\ar[r]^(0.6){\varepsilon}&\mathbb{M}_{\bullet}!}.$ In addition, by
theorem 2.4., the left adjoint
$\raisebox{.45ex}{\textup{\textbf{!`}}}$ to $!$ has the
 property that
 $(\xymatrix@1{Tot\mathcal{T_{\bullet}}{\mathbb{M}}_{\bullet}\ar[r]^(0.6){\varepsilon}&\mathbb{M}_{\bullet}!})
 \raisebox{.45ex}{\textup{\textbf{!`}}}$
 is isomorphic to
 $\xymatrix@1{Tot\mathcal{S_{\bullet}}\mathbb{M}_{\bullet}\ar[r]^(0.6){\varepsilon
 \raisebox{.45ex}{\textup{\textbf{!`}}}}&\mathbb{M}_{\bullet}},$
so $\varepsilon\raisebox{.45ex}{\textup{\textbf{!`}}}$ is a relative
quasi-isomorphism.

To see how the relative derived categories defined earlier relate to
Hochschild cohomology recall that given a presheaf of $k$-algebras
$\mathbb{A}$ the relative Hochschild cohomology of $\mathbb{A}$,
denoted $\mathbf{H}^{\bullet}(\mathbb{A},(-))$, is the same as the
relative Yoneda cohomology
$\mathbf{Ext}^{\bullet}_{\mathbb{A}-\mathbb{A}}(\mathbb{A},(-))$ of
the category of $\mathbb{A}$-bimodules. The word $\mathbf{relative}$
appears as an indication that $k$ is not necessarily a field, in
general only a commutative ring.

Thus, the relative Hochschild cohomology of a presheaf of algebras,
with coefficients in an arbitrary $\mathbb{A}$-bimodule
$\mathbb{M}$, is computed by taking any relative projective
allowable resolution of $\mathbb{A}$, applying
$Hom_{\mathbb{A}-\mathbb{A}}((-),\mathbb{M})$ to it and then taking
the homology of the resulting complex.

\begin{theorem}
$\mathbf{Ext}^{i}_{\mathbb{A}-\mathbb{A}}(\mathbb{M},\mathbb{N})\simeq
Mor_{\mathcal{D}^{-}_{k}(\mathbb{A}-\mathrm{bimod})}
(\mathbb{M}_{\bullet},\mathbb{N}_{\bullet}{[i]}).$ In particular,
$\mathbf{H}^{i}(\mathbb{A},\mathbb{N})\simeq
Mor_{\mathcal{D}^{-}_{k}(\mathbb{A}-\mathrm{bimod})}
(\mathbb{A}_{\bullet},\mathbb{N}_{\bullet}{[i]}).$
\end{theorem}
\begin{proof}

Let $Tot\mathcal{B}_{\bullet}\mathcal{S}_{\bullet}\mathbb{M}$ the
relative allowable projective resolution described in section
2.$\;$( same as  $\mathcal{U}\mathbb{M}_{\bullet}$ in this case
since $\mathbb{M}_i=0$, $(\forall) i\neq 0$.) Using proposition
3.6. and theorem 3.8. we obtain the isomorphisms\\
$\mathbf{Ext}^{i}_{\mathbb{A}-\mathbb{A}}(\mathbb{M},\mathbb{N})=
H^{i}(Hom_{\mathbb{A}-\mathbb{A}}(\mathcal{U}\mathbb{M}_{\bullet},\mathbb{N}))=
Mor_{\mathcal{K}^{-}(\mathbb{A}-\mathrm{bimod})}(\mathcal{U}\mathbb{M}_{\bullet},\mathbb{N}_{\bullet}{[i]})
$\\$\cong
Mor_{\mathcal{D}^{-}_{k}(\mathbb{A}-\mathrm{bimod})}(\mathcal{U}\mathbb{M}_{\bullet},\mathbb{N}_{\bullet}{[i]})\cong
Mor_{\mathcal{D}^{-}_{k}(\mathbb{A}-\mathrm{bimod})}(\mathbb{M}_{\bullet},\mathbb{N}_{\bullet}{[i]}).$
\end{proof}

%%%%the fourth section%%%%%%%%%
\section{Functors between derived categories}

The functor
$\xymatrix{\mathbb{A}-\mathrm{bimod}\ar[r]^{!}&\mathbb{A}!-\mathrm{bimod}}$
is exact and preserves allowability so it induces a functor between
the corresponding relative derived categories. In this section we
prove the following property of the induced functor.

\begin{theorem} The functor $\xymatrix@1{\mathcal{D}_k^{-}
({\mathbb{A}}-{\mathrm{bimod}})\ar[r]^{!}&\mathcal{D}_k^{-}
({\mathbb{A!}}-{\mathrm{bimod}})}$ is full and faithful. That is,
\begin{center}$\xymatrix{ Mor_{\mathcal{D}_k^{-}
({\mathbb{A}}-{\mathrm{bimod}})}(\mathbb{M}_{\bullet},
\mathbb{N}_{\bullet})\ar[r]^{!}&Mor_{\mathcal{D}_k^{-}
({\mathbb{A!}}-{\mathrm{bimod}})}(\mathbb{M}_{\bullet}!,
\mathbb{N}_{\bullet}!)}$\end{center} is an isomorphism of sets for all
$\mathbb{M}_\bullet,
\mathbb{N}_\bullet\in\mathcal{D}_{k}^{-}(\mathbb{A}-\mathrm{bimod})$.
\end{theorem}
The difficulties in proving the theorem  reside in two places.
First, since the morphisms in $\mathcal{D}_k^{-}
({\mathbb{A}}-{\mathrm{bimod}})$ and $\mathcal{D}_k^{-}
({\mathbb{A!}}-{\mathrm{bimod}})$ are equivalence classes of roofs,
it is not clear how one can find ancestors in $\mathcal{D}_k^{-}
({\mathbb{A}}-{\mathrm{bimod}})$ for arbitrary roofs in
$\mathcal{D}_k^{-} ({\mathbb{A!}}-{\mathrm{bimod}})$.

A good sign for that would be the existence of a left adjoint for
$!$, but there is none. Fortunately, a left adjoint exists between
${\mathbb{A}}$-${\mathrm{bimod}}$ and the full subcategory of
${\mathbb{A!}}$-${\mathrm{bimod}}$ of aligned bimodules. Second,
left adjoints do not necessarily preserve all relative
quasi-isomorphisms. However, this left adjoint preserves some that
can be used to trace back ancestors for any roof in
$Mor_{\mathcal{D}_k^{-}
(\mathbb{A}!-{\mathrm{bimod}})}(\mathbb{M}_{\bullet}!,
\mathbb{N}_{\bullet}!).$

We will prove that
$\xymatrix{\mathcal{D}^{-}_{k}(\mathbb{A}-\mathrm{bimod})\ar[r]^(.48){!}&
\mathcal{D}^{-}_{k}(\mathbb{A}!-\mathrm{albimod})}$ and the
inclusion
$\xymatrix{\mathcal{D}^{-}_{k}(\mathbb{A}!-\mathrm{albimod})\ar[r]^{inc}&
\mathcal{D}^{-}_{k}(\mathbb{A}!-\mathrm{bimod})}$ are full and
faithful.
\begin{proposition}
The functor $$\xymatrix{ \mathcal{D}_k^{-}
({\mathbb{A}}-{\mathrm{bimod}})\ar[r]^(.48){!}& \mathcal{D}_k^{-}
({\mathbb{A}}!-{\mathrm{albimod}})}$$ is full and faithful.
\end{proposition}
\begin{proof} To prove the proposition we need to show that
$$\xymatrix@1{Mor_{\mathcal{D}_k^{-}
({\mathbb{A}}-{\mathrm{bimod}})}(\mathbb{M}_{\bullet},
\mathbb{N}_{\bullet})\ar[r]^(.48){!}& Mor_{\mathcal{D}_k^{-}
({\mathbb{A!}}-{\mathrm{albimod}})}(\mathbb{M}_{\bullet}!,
\mathbb{N}_{\bullet}!)}$$ is an isomorphism for all
$\mathbb{M}_{\bullet}$ and $\mathbb{N}_{\bullet} \in
\mathcal{D}_k^{-} ({\mathbb{A}}-{\mathrm{bimod}}).$

Since for all $\mathbb{M}_{\bullet}\in \mathcal{D}_k^{-}
({\mathbb{A}}-{\mathrm{bimod}})$ there exist $\xymatrix{\mathcal{U}
\mathbb{M}_{\bullet}\ar[r]^{\varepsilon}&\mathbb{M}_{\bullet}}$
relative quasi-isomorphism in $\mathcal{D}_k^{-}
({\mathbb{A}}-{\mathrm{bimod}})$ such that $\mathcal{U}\mathbb{M}_i$
is relative projective for all $i$, we may assume that
$\mathbb{M}_{\bullet}$ is a complex of relative projective
$\mathbb{A}$ bimodules. This is because of the commutative diagram
$$\xymatrix@1{
 Mor_{\mathcal{D}_k^{-}
({\mathbb{A}}-{\mathrm{bimod}})}(\mathbb{M}_{\bullet},
\mathbb{N}_{\bullet})\ar[r]^(.48){!}\ar[d]_{\varepsilon}
            & Mor_{\mathcal{D}_k^{-}
({\mathbb{A!}}-{\mathrm{albimod}})}(\mathbb{M}_{\bullet}!,
\mathbb{N}_{\bullet}!)  \ar[d]_{\varepsilon!} \\
Mor_{\mathcal{D}_k^{-}
({\mathbb{A}}-{\mathrm{bimod}})}(\mathcal{U}\mathbb{M}_{\bullet},
\mathbb{N}_{\bullet}) \ar[r]^(.48){!} & Mor_{\mathcal{D}_k^{-}
({\mathbb{A!}}-{\mathrm{albimod}})}(\mathcal{U}\mathbb{M}_{\bullet}!,
\mathbb{N}_{\bullet}!) } $$ where $\varepsilon$ and $\varepsilon!$
are isomorphisms.

Because $(\mathbb{M}_i)^{p}$ is a relative projective
$\mathbb{A}^{p}$-bimodule, $(\forall) p\in\mathcal{C}$, each
$\mathbb{M}_i!$ admits a resolution of relative projective aligned
$\mathbb{A!}$-bimodules obtained using $\mathcal{T}_{\bullet}$. The
total complex of the double complex obtained by taking the
resolution of each $\mathbb{M}_i!$ gives a relative
quasi-isomorphism $\xymatrix@1{Tot(\mathcal{T_{\bullet}}
{\mathbb{M}}_{\bullet})\ar[r]^(.6){\varepsilon}&\mathbb{M}_{\bullet}!,}$where
each $Tot(\mathcal{T_{\bullet}}{\mathbb{M}}_{\bullet})_i$ is a
relative projective aligned $\mathbb{A!}$ bimodule.

Moreover, the left adjoint $\raisebox{.45ex}{\textup{\textbf{!`}}}$
has the property that
 $(\xymatrix@1{Tot\mathcal{T_{\bullet}}{\mathbb{M}}_{\bullet}
 \ar[r]^(.65){\varepsilon}&\mathbb{M}_{\bullet}!})\raisebox{.45ex}{\textup{\textbf{!`}}}$
is isomorphic to
$\xymatrix@1{Tot\mathcal{S_{\bullet}}\mathbb{M}_{\bullet}\ar[r]^(.6){\varepsilon
\raisebox{.45ex}{\textup{\textbf{!`}}}}&\mathbb{M}_{\bullet}}$ and
$\varepsilon\raisebox{.45ex}{\textup{\textbf{!`}}}$ is a relative
quasi-isomorphism.
Now, given any roof $$\xymatrix{&X_{\bullet}\ar[dl]_{s}\ar[dr]^f\\
\mathbb{M}_{\bullet}!&&\mathbb{N}_{\bullet}!}$$ in
$Mor_{\mathcal{D}_k^{-}
({\mathbb{A!}}-{\mathrm{albimod}})}(\mathbb{M}_{\bullet}!,
\mathbb{N}_{\bullet}!)$ take
$\xymatrix@1{Tot(\mathcal{T_{\bullet}}{\mathbb{M}}_{\bullet})\ar[r]^(.65){\varepsilon}&\mathbb{M}_{\bullet}!}$
as above.

By applying $Mor_{\mathcal{D}_k^{-}
({\mathbb{A!}}-{\mathrm{albimod}})}(Tot(\mathcal{T_{\bullet}}{\mathbb{M}}_{\bullet}),(-)
)$ to the distinguished triangle $$\xymatrix{
X_{\bullet}\ar[r]^s&\mathbb{M}_{\bullet}!\ar[r]^{}&\mathcal{C}(s)_{\bullet}\ar[r]^{}&X_{\bullet}[1]}$$
we obtain a long exact sequence.\\ In this sequence
$Mor_{\mathcal{D}_k^{-}}(Tot(\mathcal{T_{\bullet}}{\mathbb{M}}_{\bullet}),\mathcal{C}(s)_{\bullet})=0$
because $\mathcal{C}(s)_{\bullet}$ is contractible, as a complex of
$k$-bimodules, and
$Tot(\mathcal{T_{\bullet}}{\mathbb{M}}_{\bullet})$ is a complex of
relative projective aligned $\mathbb{A}!$ bimodules, so the map
$$\xymatrix{Mor_{\mathcal{D}_k^{-}}(Tot(\mathcal{T_{\bullet}}{\mathbb{M}}_{\bullet}),X)
\ar[r]^(.48)s&
Mor_{\mathcal{D}_k^{-}}(Tot(\mathcal{T_{\bullet}}{\mathbb{M}}_{\bullet}),
\mathbb{M}_\bullet! )}$$ is onto. Because $\varepsilon\in
Mor_{\mathcal{D}_k^{-}
({\mathbb{A!}}-{\mathrm{albimod}})}(Tot(\mathcal{T_{\bullet}}{\mathbb{M}}_{\bullet})
,\mathbb{M}_{\bullet}!),$ there exist\\ $q\in
$$Mor_{\mathcal{D}_k^{-}
({\mathbb{A!}}-{\mathrm{albimod}})}(Tot(\mathcal{T_{\bullet}}{\mathbb{M}}_{\bullet}),X_{\bullet})$
such that the diagram $$ \xymatrix{&
Tot(\mathcal{T_{\bullet}}{\mathbb{M}}_{\bullet})\ar[dl]_q\ar[d]^{\varepsilon}\\
X_{\bullet}\ar[r]^s&\mathbb{M}_{\bullet}!}$$ commutes. The map $q$
is a relative quasi-isomorphism because both $s$ and $\varepsilon$
are and we have the equivalence of roofs
$$\xymatrix{&X_{\bullet}\ar[dl]_{s}\ar[dr]^f\\
\mathbb{M}_{\bullet}!&&\mathbb{N}_{\bullet}!} \mathrm{and}
\xymatrix{
&Tot(\mathcal{T_{\bullet}}{\mathbb{M}}_{\bullet})\ar[dl]_{\varepsilon}\ar[dr]^{fq}\\
\mathbb{M}_{\bullet}!&&\mathbb{N}_{\bullet}! }$$ because the diagram
$$\xymatrix{&&Tot(\mathcal{T_{\bullet}}{\mathbb{M}}_{\bullet})\ar[dl]_(.57)q\ar[dr]^{id}\\
&X_{\bullet}\ar[dl]_s\ar[drrr]^(.6)f
&&Tot(\mathcal{T_{\bullet}}{\mathbb{M}}_{\bullet})\ar[dlll]_(.6){\varepsilon}
\ar[dr]^{fq}\\
\mathbb{M_{\bullet}!}&&&&\mathbb{N_{\bullet}!} }$$ is commutative.

Since
$(\xymatrix{Tot\mathcal{T_{\bullet}}{\mathbb{M}}_{\bullet}\ar[r]^(.6){\varepsilon}&\mathbb{M}_{\bullet}!})
\raisebox{.45ex}{\textup{\textbf{!`}}}$ is isomorphic to
$\xymatrix{Tot\mathcal{S_{\bullet}}\mathbb{M}_{\bullet}\ar[r]^(.6){\varepsilon\raisebox{.45ex}{\textup{\textbf{!`}}}}
&\mathbb{M}_{\bullet}}$ and
$\varepsilon\raisebox{.45ex}{\textup{\textbf{!`}}}$ is a relative
quasi-isomorphism, the roof $$\xymatrix{ &
(Tot(\mathcal{T}_{\bullet}\mathbb{M}_{\bullet}))\raisebox{.45ex}{\textup{\textbf{!`}}}\ar[dl]_(.6){\varepsilon_{\mathbb{M}_{\bullet}}
\varepsilon\raisebox{.45ex}{\textup{\textbf{!`}}}}\ar[dr]^(.6){\varepsilon_{\mathbb{N}_{\bullet}}f\raisebox{.45ex}{\textup{\textbf{!`}}}
q\raisebox{.45ex}{\textup{\textbf{!`}}}}\\
\mathbb{M}_{\bullet}&&\mathbb{N}_{\bullet}}$$ exists in
$\mathcal{D}_k^{-} ({\mathbb{A}}-{\mathrm{bimod}}).$\\ Here,
$\varepsilon_{\mathbb{M}_{\bullet}}$ and
$\varepsilon_{\mathbb{N}_{\bullet}}$ are the maps of complexes
induced by the counit of the adjunction
$\xymatrix{\mathbb{A}-\mathrm{bimod}\ar@<1ex>[r]^(.46){!}
&\mathbb{A!}-\mathrm{albimod.}\ar@<1ex>[l]^{\raisebox{.45ex}{\textup{\textbf{!`}}}}}$
The image of this roof via ! is
$$\xymatrix{&[(Tot(\mathcal{T_{\bullet}}{\mathbb{M}}_{\bullet})\raisebox{.45ex}{\textup{\textbf{!`}}}\ar[dl]_(.6){\varepsilon_{\mathbb{M}_{\bullet}}
!\varepsilon\raisebox{.45ex}{\textup{\textbf{!`}}}!}\ar[dr]^(.6){\varepsilon_{\mathbb{N}_{\bullet}}!f\raisebox{.45ex}{\textup{\textbf{!`}}}
!q\raisebox{.45ex}{\textup{\textbf{!`}}}!}])!\\
\mathbb{M}_{\bullet}!&&\mathbb{N}_{\bullet}! }$$ and is equivalent
to $$\xymatrix{
&Tot(\mathcal{T_{\bullet}}{\mathbb{M}}_{\bullet})\ar[dl]_{\varepsilon}\ar[dr]^{fq}\\
\mathbb{M}_{\bullet}!&&\mathbb{N}_{\bullet}! }. $$ This results from
the commutative diagram
$$\xymatrix@1@=9pt@M=9pt{&&Tot(\mathcal{T_{\bullet}}{\mathbb{M}}_{\bullet})\ar[dl]_{id}
\ar[dr]^(.6){\eta_{Tot(\mathcal{T_{\bullet}}{\mathbb{M}}_{\bullet})}}\\
&Tot(\mathcal{T_{\bullet}}{\mathbb{M}}_{\bullet})\ar[dl]_{\varepsilon}\ar[drrr]^(.6){fq}&&
[(Tot(\mathcal{T_{\bullet}}{\mathbb{M}}_{\bullet}))\raisebox{.45ex}{\textup{\textbf{!`}}}]!\ar[dlll]_(.6){\varepsilon_{\mathbb{M}_{\bullet}}
!\varepsilon\raisebox{.45ex}{\textup{\textbf{!`}}}!}\ar[dr]^(.7){\varepsilon_{\mathbb{N}^{\bullet}}!f\raisebox{.45ex}{\textup{\textbf{!`}}}!
q\raisebox{.45ex}{\textup{\textbf{!`}}}!}\\
\mathbb{M}_{\bullet}!&&&&\mathbb{N}_{\bullet}!}$$
(1)$\;$$\varepsilon_{\mathbb{M}_{\bullet}}!(\varepsilon\raisebox{.45ex}{\textup{\textbf{!`}}})!
\eta_{Tot(\mathcal{T_{\bullet}}{\mathbb{M}}_{\bullet})}=\varepsilon$
and\\
(2)$\;$$\varepsilon_{\mathbb{N}_{\bullet}}![f\raisebox{.45ex}{\textup{\textbf{!`}}}
q\raisebox{.45ex}{\textup{\textbf{!`}}}]!
\eta_{Tot(\mathcal{T_{\bullet}}{\mathbb{M}}_{\bullet})}=fq$\\
To check (1) observe that we have
$\varepsilon_{\mathbb{M}_{\bullet}}!\eta_{\mathbb{M}_{\bullet}!}=id_{\mathbb{M}_{\bullet}!}$
by the adjunction. In addition, the functoriality of $\eta$ induces
the commutative square
$$\xymatrix{Tot(\mathcal{T_{\bullet}}{\mathbb{M}}_{\bullet})
\ar[d]_{\eta_{Tot(\mathcal{T_{\bullet}}{\mathbb{M}}_{\bullet})}}
\ar[r]^{\varepsilon}&\mathbb{M}_{\bullet}!\ar[d]^{\eta_{\mathbb{M}_{\bullet}!}}\\
[(Tot(\mathcal{T_{\bullet}}{\mathbb{M}}_{\bullet}))\raisebox{.45ex}{\textup{\textbf{!`}}}]!
\ar[r]^(.6){(\varepsilon\raisebox{.45ex}{\textup{\textbf{!`}}})!}&[(\mathbb{M}_{\bullet}!)\raisebox{.45ex}{\textup{\textbf{!`}}}]!}$$
Thus, we have
$(\varepsilon\raisebox{.45ex}{\textup{\textbf{!`}}})!\eta_{Tot(\mathcal{T_{\bullet}}{\mathbb{M}}_{\bullet})}
=\eta_{\mathbb{M}_{\bullet}!}\varepsilon $ and by composing with
$\varepsilon_{\mathbb{M}_{\bullet}}!$  we obtain (1). Similarly one
may check (2).

To prove injectivity, let $$\xymatrix{&\mathbb{R}_{\bullet}!\ar[dl]_{r!}\ar[dr]^{f!}\\
\mathbb{M}_{\bullet}!&&\mathbb{N}_{\bullet}!}\mathrm{and}
\xymatrix{&\mathbb{S}_{\bullet}!\ar[dl]_{s!}\ar[dr]^{g!}\\
\mathbb{M}_{\bullet}!&&\mathbb{N}_{\bullet}!}$$ be equivalent roofs in
$\mathcal{D}_k^{-} ({\mathbb{A!}}-{\mathrm{albimod}}).$ One may assume that $\mathbb{R}_{\bullet}$ is a complex or relative
projective $\mathbb{A}$ bimodules. To see this, let
$\xymatrix{\mathcal{U}\mathbb{M}_{\bullet}\ar[r]^{\varepsilon}&\mathbb{M}_{\bullet}}$
the relative quasi-isomorphism with  $\mathcal{U}\mathbb{M}_i$
relative projective $\mathbb{A}$-bimodules.

Again, applying $Mor_{\mathcal{D}_k^{-}
({\mathbb{A}}-{\mathrm{bimod}})}(\mathcal{U}\mathbb{M}_{\bullet},(-))$
to the distinguished triangle
$\xymatrix{\mathbb{R}_{\bullet}\ar[r]^{r}&\mathbb{M}_{\bullet}\ar[r]&\mathcal{C}(r)_{\bullet}
\ar[r]&\mathbb{R}[1]_{\bullet},}$ in $\mathcal{D}_k^{-}
({\mathbb{A}}-{\mathrm{bimod}}),$ we obtain a long exact sequence
where
$Mor_{\mathcal{D}_k^{-}}(\mathcal{U}\mathbb{M}_{\bullet},\mathcal{C}(r)_{\bullet})=0.$

This implies the existence of a map $t$ such that the following
diagram $$\xymatrix{&\mathcal{U}\mathbb{M}_{\bullet}\ar[dl]_{t}\ar[d]^{\varepsilon}\\
\mathbb{R}_{\bullet}\ar[r]^{r}&\mathbb{M}_{\bullet}}$$ commutes. In
addition, $t$ is a relative quasi-isomorphism, since $r$ and
$\varepsilon$ are and we have the equivalent roofs $$\xymatrix{&\mathbb{R}_{\bullet}\ar[dl]_{r}\ar[dr]^{f}\\
\mathbb{M}_{\bullet}&&\mathbb{N}_{\bullet}}\mathrm{and}\xymatrix{&\mathcal{U}\mathbb{M}_{\bullet}\ar[dl]_{\varepsilon}\ar[dr]^{ft}\\
\mathbb{M}_{\bullet}&&\mathbb{N}_{\bullet}}$$ in $\mathcal{D}_k^{-}
({\mathbb{A}}-{\mathrm{bimod}})$ because of the following
commutative diagram $$
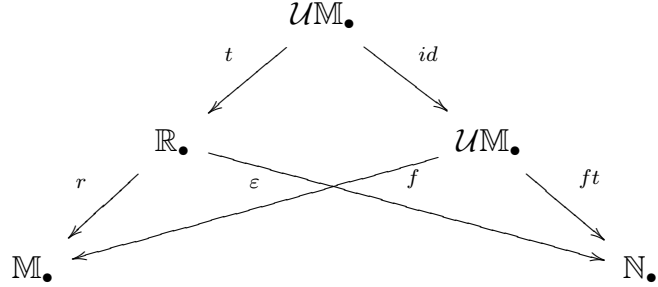
\xymatrix{&&\mathcal{U}\mathbb{M}_{\bullet}\ar[dl]_{t}\ar[dr]^{id}\\
&\mathbb{R}_{\bullet}\ar[dl]_{r}\ar[drrr]^{f}&&\mathcal{U}\mathbb{M}_{\bullet}\ar[dlll]_{
\varepsilon}\ar[dr]^{ft}\\
\mathbb{M}_{\bullet}&&&&\mathbb{N}_{\bullet}}$$  This implies the
the equivalence of $$
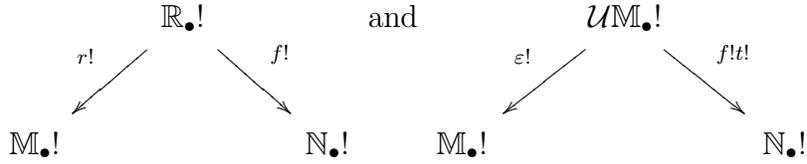
\xymatrix{&\mathbb{R}_{\bullet}!\ar[dl]_{r!}\ar[dr]^{f!}\\
\mathbb{M}_{\bullet}!&&\mathbb{N}_{\bullet}!}\mathrm{and}
\xymatrix{&\mathcal{U}\mathbb{M}_{\bullet}!\ar[dl]_{\varepsilon!}\ar[dr]^{f!t!}\\
\mathbb{M}_{\bullet}!&&\mathbb{N}_{\bullet}!}$$ in
$\mathcal{D}_k^{-} ({\mathbb{A!}}-{\mathrm{albimod}}).$ So, we may
assume that $\mathbb{R}_{\bullet}$ is a complex of relative
projective $\mathbb{A}$-bimodules. The equivalence of \begin{center}$
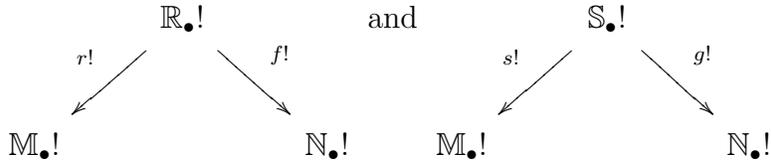
\xymatrix{&\mathbb{R}_{\bullet}!\ar[dl]_{r!}\ar[dr]^{f!}\\
\mathbb{M}_{\bullet}!&&\mathbb{N}_{\bullet}!}\mathrm{and}
\xymatrix{&\mathbb{S}_{\bullet}!\ar[dl]_{s!}\ar[dr]^{g!}\\
\mathbb{M}_{\bullet}!&&\mathbb{N}_{\bullet}!}$\end{center} translates into the
existence of a commutative diagram
$$
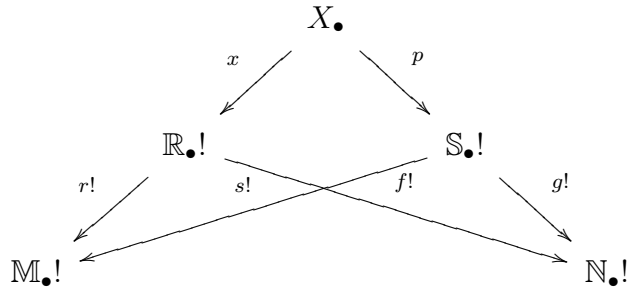
\xymatrix{&&X_{\bullet}\ar[dl]_{x}\ar[dr]^{p}\\
&\mathbb{R}_{\bullet}!\ar[dl]_{r!}\ar[drrr]^{f!}&&\mathbb{S}_{\bullet}!\ar[dlll]_{s!}\ar[dr]^
{g!}\\
\mathbb{M}_{\bullet}!&&&&\mathbb{N}_{\bullet}!}$$  Here,
$X_{\bullet}\in\mathcal{D}_k^{-} ({\mathbb{A!}}-{\mathrm{albimod}})$
and $x$ is a relative quasi-isomorphism such that $f!x=g!p$, (1) and
$s!p=r!x$, (2). Since $x$ is a  relative quasi-isomorphism and
$Tot\mathcal{T}_{\bullet}\mathbb{R}_{\bullet}$ is a complex of
aligned relative projective $\mathbb{A}!$-bimodules there exist $j$
such that the diagram
\begin{center}$\xymatrix{&Tot\mathcal{T}_{\bullet}\mathbb{R}_{\bullet}\ar[dl]_{j}
\ar[d]^{\varepsilon}\\
X_{\bullet}\ar[r]^{x}&\mathbb{R}_{\bullet}!}$\end{center}  is commutative.
Moreover, $j$ is a relative quasi-isomorphism because $\varepsilon$
and $x$ are. We obtain the commutative diagram
\begin{center}$\xymatrix{&&Tot\mathcal{T}_{\bullet}\mathbb{R}_{\bullet}\ar[dl]_{xj}\ar[dr]^{pj}\\
&\mathbb{R}_{\bullet}!\ar[dl]_{r!}\ar[drrr]^(.57){f!}&&\mathbb{S}_{\bullet}!\ar[dlll]_(.57){s!}\ar[dr]^
{g!}\\
\mathbb{M}_{\bullet}!&&&&\mathbb{N}_{\bullet}!}$\end{center} because
$f!xj=g!pj$, by (1) and $s!pj=r!xj $, by (2). Because ! is full
and faithful we have the isomorphism
$\xymatrix{(\mathbb{T}_{\bullet}!)\raisebox{.45ex}{\textup{\textbf{!`}}}\ar[r]^{\varepsilon_{\mathbb{T}_{\bullet}}}
&\mathbb{T}_{\bullet}}$ for all $\mathbb{T}_{\bullet}$ in
$\mathcal{D}_k^{-} ({\mathbb{A}}-{\mathrm{bimod}})$ and so
$(r!)\raisebox{.45ex}{\textup{\textbf{!`}}}$ and
$(s!)\raisebox{.45ex}{\textup{\textbf{!`}}}$ are relative
quasi-isomorphisms in $\mathcal{D}_k^{-}
({\mathbb{A}}-{\mathrm{bimod}}).$  In addition,
$\varepsilon\raisebox{.45ex}{\textup{\textbf{!`}}}$ is a relative
quasi-isomorphism and we get the commutative diagram\newpage
$$
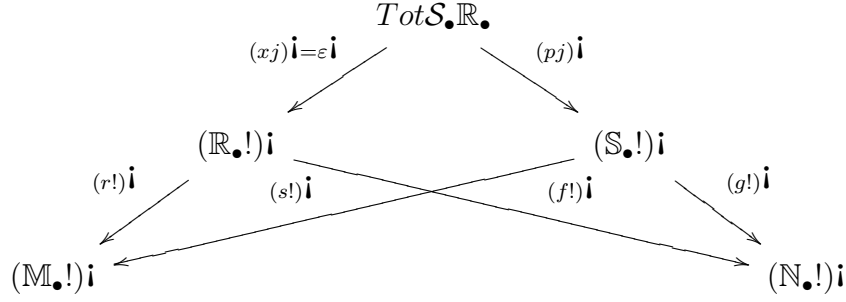
\xymatrix{&&Tot\mathcal{S}_{\bullet}\mathbb{R}_{\bullet}\ar[dl]_{(xj)\raisebox{.45ex}{\textup{\textbf{!`}}}
=\varepsilon\raisebox{.45ex}{\textup{\textbf{!`}}}}\ar[dr]^{(pj)\raisebox{.45ex}{\textup{\textbf{!`}}}
}\\
&(\mathbb{R}_{\bullet}!)\raisebox{.45ex}{\textup{\textbf{!`}}}\ar[dl]_{(r!)\raisebox{.45ex}{\textup{\textbf{!`}}}
}\ar[drrr]^(.57){(f!)\raisebox{.45ex}{\textup{\textbf{!`}}}}&&
(\mathbb{S}_{\bullet}!)\raisebox{.45ex}{\textup{\textbf{!`}}}\ar[dlll]_(.57){(s!)\raisebox{.45ex}{\textup{\textbf{!`}}}
}\ar[dr]^{(g!)\raisebox{.45ex}{\textup{\textbf{!`}}}}\\
(\mathbb{M}_{\bullet}!)\raisebox{.45ex}{\textup{\textbf{!`}}}&&&&
(\mathbb{N}_{\bullet}!)\raisebox{.45ex}{\textup{\textbf{!`}}}}$$
Finally, we
obtain the equivalence of $$
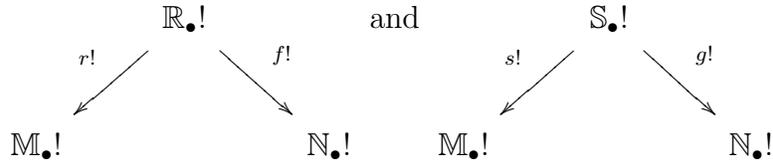
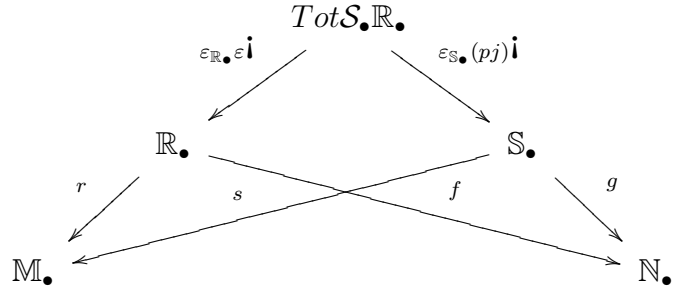
\xymatrix{&\mathbb{R}_{\bullet}!\ar[dl]_{r!}\ar[dr]^{f!}\\
\mathbb{M}_{\bullet}!&&\mathbb{N}_{\bullet}!}\mathrm{and}
\xymatrix{&\mathbb{S}_{\bullet}!\ar[dl]_{s!}\ar[dr]^{g!}\\
\mathbb{M}_{\bullet}!&&\mathbb{N}_{\bullet}!}$$ by constructing
$$\xymatrix{&&Tot\mathcal{S}_{\bullet}\mathbb{R}_{\bullet}
\ar[dl]_{\varepsilon_{\mathbb{R}_{\bullet}}\varepsilon\raisebox{.45ex}{\textup{\textbf{!`}}}}
\ar[dr]^{\varepsilon_{\mathbb{S}_{\bullet}}(pj)\raisebox{.45ex}{\textup{\textbf{!`}}}}\\
&\mathbb{R}_{\bullet}\ar[dl]_{r}\ar[drrr]^(.57){f}&&
\mathbb{S}_{\bullet}\ar[dlll]_(.57){s}\ar[dr]^{g}\\
\mathbb{M}_{\bullet}&&&&\mathbb{N}_{\bullet}}$$  This is because
$f\varepsilon_{\mathbb{R}_{\bullet}}\varepsilon\raisebox{.45ex}{\textup{\textbf{!`}}}
=\varepsilon_{\mathbb{N}_{\bullet}}(f!)\raisebox{.45ex}{\textup{\textbf{!`}}}
\varepsilon\raisebox{.45ex}{\textup{\textbf{!`}}}=
\varepsilon_{\mathbb{N}_{\bullet}}(g!)\raisebox{.45ex}{\textup{\textbf{!`}}}(pj)\raisebox{.45ex}{\textup{\textbf{!`}}}
=g\varepsilon_{\mathbb{S}_{\bullet}}(pj)\raisebox{.45ex}{\textup{\textbf{!`}}}$
and\\
$s\varepsilon_{\mathbb{S}_{\bullet}}(pj)\raisebox{.45ex}{\textup{\textbf{!`}}}
=\varepsilon_{\mathbb{M}_{\bullet}}(s!)\raisebox{.45ex}{\textup{\textbf{!`}}}(pj)\raisebox{.45ex}{\textup{\textbf{!`}}}
=
\varepsilon_{\mathbb{M}_{\bullet}}(r!)\raisebox{.45ex}{\textup{\textbf{!`}}}
\varepsilon\raisebox{.45ex}{\textup{\textbf{!`}}}=
r\varepsilon_{\mathbb{R}_{\bullet}}\varepsilon\raisebox{.45ex}{\textup{\textbf{!`}}}.$
\end{proof}

We show now that the inclusion $$\xymatrix{\mathcal{D}_k^{-}
({\mathbb{A}}!-{\mathrm{albimod}})\ar[r]^{inc}&\mathcal{D}_k^{-}
({\mathbb{A}}!-{\mathrm{bimod}})}$$ is full and faithful. The lack
of an adjoint in this case requires a two step process of replacing
the top of each roof by a complex of aligned bimodules. For
$X\in\mathbb{A!}-{\mathrm{bimod}}$, let
$X^{+}:=\prod_{i\in\mathcal{C}}\varphi^{ii}X$. This defines an exact
functor
$\xymatrix{\mathbb{A!}-{\mathrm{bimod}}\ar[r]^{+}&\mathbb{A!}-{\mathrm{bimod}}}$
that preserves allowability, so also relative quasi-isomorphisms.\\
We also have the natural maps
$\xymatrix{X\ar[r]^{\beta_{X}}&X^{+}}$,
$\xymatrix{x\ar[r]&<\varphi^{ii}x>}$ and
$\xymatrix{X_{al}\ar[r]^{\gamma_{X}}&X^{+}}$,$\xymatrix{<x_{ij}>\ar[r]&<\sum_{j\geq
i}x_{ij}>.}$ \\
Also, if $X$ is aligned both $\beta_{X}^{}$ and $\gamma_{X}^{}$ are
isomorphisms and $\beta_{X}^{}=\gamma_{X}^{}\alpha_{X}^{}$, where
$\alpha$ is the natural isomorphism $\alpha :
Id_{\mathbb{A!}-albimod}\longrightarrow (-)_{al}\circ inc $.

\begin{proposition}
The functor $$\xymatrix{\mathcal{D}_k^{-}
({\mathbb{A}}!-{\mathrm{albimod}})\ar[r]^{inc}&\mathcal{D}_k^{-}
({\mathbb{A}}!-{\mathrm{bimod}})}$$ is full and faithful.
\end{proposition}
\begin{proof}We have to prove that
$$\xymatrix{Mor_{\mathcal{D}_k^{-}
({\mathbb{A}}!-{\mathrm{albimod}})}(M_{\bullet},
N_{\bullet})\ar[r]^{inc}& Mor_{\mathcal{D}_k^{-}
({\mathbb{A!}}-{\mathrm{bimod}})}(M_{\bullet}, N_{\bullet})}$$ is an
isomorphism of sets for all $M_{\bullet}$ and $N_{\bullet} \in
\mathcal{D}_k^{-} ({\mathbb{A!}}-{\mathrm{albimod}})$. First, we
prove that the map is onto. For any roof
$$\xymatrix{&X_{\bullet}\ar[dl]_{s}\ar[dr]^{f}\\M_{\bullet}&&N_{\bullet}}$$
in
$Mor_{\mathcal{D}_{k}^{-}(\mathbb{A!}-\mathrm{bimod})}(M_{\bullet},N_{\bullet})$
we have the equivalences
$$\xymatrix{&X_{\bullet}\ar[dl]_{s}\ar[dr]^{f}\\M_{\bullet}&&N_{\bullet}}\mathrm{and}
\xymatrix{&X^{+}_{\bullet}\ar[dl]_{\beta_{M_{\bullet}}^{-1}s^{+}}
\ar[dr]^{\beta_{N_{\bullet}}^{-1}f^{+}}\\M_{\bullet}&&N_{\bullet}}$$
$$\xymatrix{&X^{+}_{\bullet}\ar[dl]_{\beta_{M_{\bullet}}^{-1}s^{+}}
\ar[dr]^{\beta_{N_{\bullet}}^{-1}f^{+}}\\M_{\bullet}&&N_{\bullet}}\mathrm{and}
 \xymatrix{&X_{\bullet
al}\ar[dl]_{\alpha_{M_{\bullet}}^{-1}s_{al}}
\ar[dr]^{\alpha_{N_{\bullet}}^{-1}f_{al}}\\M_{\bullet}&&N_{\bullet}}$$

To see this, observe that since $\beta$ is a natural transformation
we have $s^{+}\beta_{X_{\bullet}}^{}=\beta_{M_{\bullet}}^{}s$ and
$f^{+}\beta_{X_{\bullet}}^{}=\beta_{N_{\bullet}}^{}f$.

In addition, because $M_{\bullet}$ and $N_{\bullet}$ are aligned
$\beta_{M_{\bullet}}^{}$ and $\beta_{N_{\bullet}}^{}$ are
isomorphisms and we obtain
$\beta_{M_{\bullet}}^{-1}s^{+}\beta_{X_{\bullet}}^{}=s$ and
$\beta_{N_{\bullet}}^{-1}f^{+}\beta_{X_{\bullet}}^{}=f$.

This implies the first equivalence because the diagram
\begin{center}$\xymatrix{&&X_{\bullet}\ar[dl]_{id}\ar[dr]^{\beta_{X_{\bullet}}^{}}\\
&X_{\bullet}\ar[dl]_{s}\ar[drrr]^(.57){f}&&X^{+}_{\bullet}\ar[dlll]_(.57){\beta_{M_{\bullet}}^{-1}s^{+}}
\ar[dr]^{\beta_{N_{\bullet}}^{-1}f^{+}}\\
M_{\bullet}&&&&N_{\bullet}}$\end{center} is commutative.

For the second equivalence, since $\gamma$ is natural we have
$s^{+}\gamma_{X_{\bullet}}^{}=\gamma_{M_{\bullet}}^{}s_{al}$ and
$f^{+}\gamma_{X_{\bullet}}^{}=\gamma_{N_{\bullet}}^{}f_{al}$.

Because $M_{\bullet}$ and $N_{\bullet}$ are aligned
$\gamma_{M_{\bullet}}^{}, \gamma_{N_{\bullet}}^{},
\alpha_{M_{\bullet}}^{}, \alpha_{N_{\bullet}}^{},
\beta_{M_{\bullet}}^{}$ and $\beta_{N_{\bullet}}^{}$ are
isomorphisms, so we get
$\beta_{M_{\bullet}}^{-1}s^{+}\gamma_{X_{\bullet}}^{}=
\beta_{M_{\bullet}}^{-1}\gamma_{M_{\bullet}}^{}s_{al}=\alpha_{M_{\bullet}}^{-1}
\gamma_{M_{\bullet}}^{-1}\gamma_{M_{\bullet}}^{}s_{al}=\alpha_{M_{\bullet}}^{-1}s_{al}$
and $\beta_{N_{\bullet}}^{-1}f^{+}\gamma_{X_{\bullet}}^{}=
\beta_{N_{\bullet}}^{-1}\gamma_{N_{\bullet}}^{}f_{al}=\alpha_{N_{\bullet}}^{-1}
\gamma_{N_{\bullet}}^{-1}\gamma_{N_{\bullet}}^{}f_{al}=\alpha_{N_{\bullet}}^{-1}f_{al}$.

The diagram
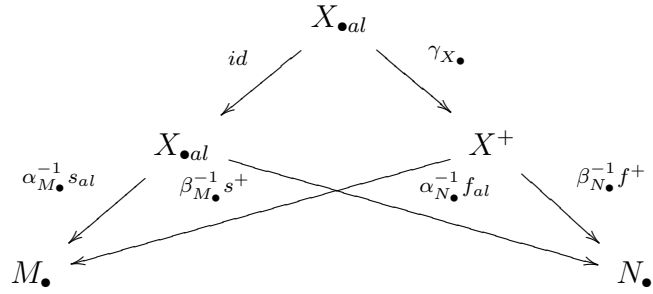
\begin{center}$\xymatrix{&&X_{\bullet
al}\ar[dl]_{id}\ar[dr]^{\gamma_{X_{\bullet}}^{}}\\
&X_{\bullet al}\ar[dl]_{\alpha_{M_{\bullet}}^{-1}s_{al}}
\ar[drrr]^(.57){\alpha_{N_{\bullet}}^{-1}f_{al}}&&
X^{+}\ar[dlll]_(.57){\beta_{M_{\bullet}}^{-1}s^{+}}\ar[dr]^{\beta_{N_{\bullet}}^{-1}f^{+}}\\
M_{\bullet}&&&&N_{\bullet}}$\end{center} is commutative and implies the second
equivalence.  Now, the surjectivity follows since the roof
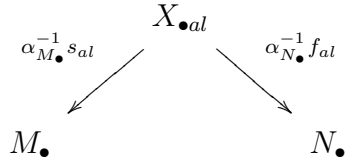
\begin{center}$\xymatrix{&X_{\bullet al}\ar[dl]_{\alpha_{M_{\bullet}}^{-1}s_{al}}
\ar[dr]^{\alpha_{N_{\bullet}}^{-1}f_{al}}\\
M_{\bullet}&&N_{\bullet}}$\end{center} exists in
$Mor_{\mathcal{D}_{k}^{-}(\mathbb{A!}-\mathrm{albimod})}(M_{\bullet},N_{\bullet})$
and its image is equivalent to
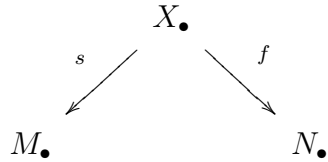
\begin{center}$\xymatrix{&X_{\bullet}\ar[dl]_{s}\ar[dr]^{f}\\M_{\bullet}&&N_{\bullet}}$\end{center}
in
$Mor_{\mathcal{D}_{k}^{-}(\mathbb{A!}-\mathrm{bimod})}(M_{\bullet},N_{\bullet})$.

To prove the injectivity, let
$$
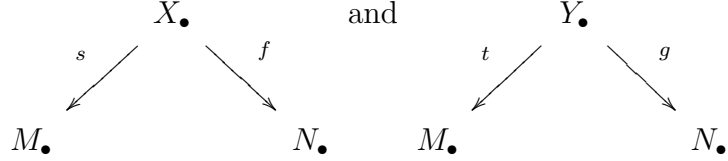
\xymatrix{&X_{\bullet}\ar[dl]_{s}\ar[dr]^{f}\\M_{\bullet}&&N_{\bullet}}\mathrm{and}
\xymatrix{&Y_{\bullet}\ar[dl]_{t}\ar[dr]^{g}\\M_{\bullet}&&N_{\bullet}}$$
in
$Mor_{\mathcal{D}_{k}^{-}(\mathbb{A!}-\mathrm{albimod})}(M_{\bullet},N_{\bullet})$
equivalent in
$Mor_{\mathcal{D}_{k}^{-}(\mathbb{A!}-\mathrm{bimod})}(M_{\bullet},N_{\bullet})$.\\
Thus, we have a commutative diagram
$$
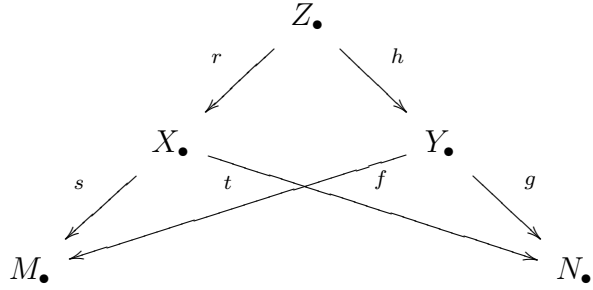
\xymatrix{&&Z_{\bullet}\ar[dl]_{r}\ar[dr]^{h}\\
&X_{\bullet}\ar[dl]_{s}\ar[drrr]^{f}&&Y_{\bullet}\ar[dlll]_{t}\ar[dr]^{g}\\
M_{\bullet}&&&&N_{\bullet}}$$ where $r$ and $ s$ are relative
quasi-isomorphisms. Since the alignment functor preserves relative
quasi-isomorphisms and $M_{\bullet}$, $N_{\bullet}$, $X_{\bullet}$
and $Y_{\bullet}$ are complexes of aligned $\mathbb{A!}$ bimodules
we have the commutative diagram $$
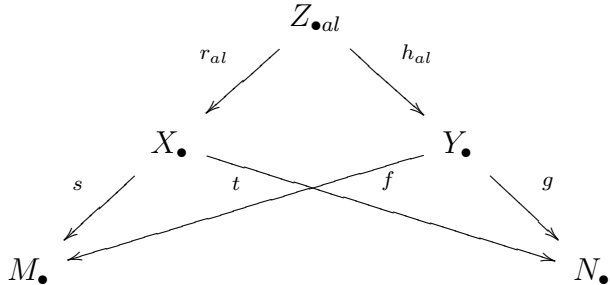
\xymatrix{&&Z_{\bullet
al}\ar[dl]_{r_{al}}\ar[dr]^{h_{al}}\\
&X_{\bullet}\ar[dl]_{s}\ar[drrr]^{f}&&Y_{\bullet}\ar[dlll]_{t}\ar[dr]^{g}\\
M_{\bullet}&&&&N_{\bullet}}$$ which implies the equivalence of roofs
$$\xymatrix{&X_{\bullet}\ar[dl]_{s}\ar[dr]^{f}\\M_{\bullet}&&N_{\bullet}}
\xymatrix{&Y_{\bullet}\ar[dl]_{t}\ar[dr]^{g}\\M_{\bullet}&&N_{\bullet}}$$
in
$Mor_{\mathcal{D}_{k}^{-}(\mathbb{A!}-\mathrm{albimod})}(M_{\bullet},N_{\bullet})$,
and so the injectivity of $inc$.

\end{proof}

The proof of  theorem 4.1. follows now easily combining propositions
4.2. and 4.3. In particular, we obtain the following theorem of [2],
due to M. Gerstenhaber and S. D. Schack.

\begin{corollary}{(Special Cohomology Comparison Theorem)}\\
The functor ! induces an isomorphism of relative Yoneda cohomologies
$$Ext^{\bullet}_{\mathbb{A}-\mathbb{A}}((-),(-))\cong
Ext^{\bullet}_{\mathbb{A}!-\mathbb{A}!}((-)!,(-)!).$$ In particular,
we have an isomorphism of relative Hochschild cohomologies
$$H^{\bullet}(\mathbb{A},(-))\cong H^{\bullet}(\mathbb{A}!,(-)!).$$
\end{corollary}
\begin{proof}
$$Ext^{i}_{\mathbb{A}-\mathbb{A}}(\mathbb{M},\mathbb{N})\cong\
Mor_{\mathcal{D}^{-}_{k}(\mathbb{A}-\mathrm{bimod})}
(\mathbb{M}_{\bullet},\mathbb{N}_{\bullet}{[i]})\cong$$
$$\cong Mor_{\mathcal{D}^{-}_{k}(\mathbb{A}!-\mathrm{bimod})}
(\mathbb{M}_{\bullet}!,\mathbb{N}_{\bullet}{[i]}!)\cong
Ext^{i}_{\mathbb{A}!-\mathbb{A}!}(\mathbb{M}!,\mathbb{N}!).$$
\end{proof}

$\mathbf{Note:}$ By taking a very different approach Wendy Lowen and
Michel Van Den Bergh also proved in [8] that the functor ! is full
and faithful.

\appendix
\section{theorem 2.4.}

\begin{theorem} 1. The functor
$!:\mathbb{A}$-bimod$\longrightarrow\mathbb{A}!$-albimod admits a
left adjoint
$\raisebox{.45ex}{\textup{\textbf{!`}}}:\mathbb{A}!$-albimod$\longrightarrow\mathbb{A}$-bimod.\\
2. There are natural isomorphisms
$\mathcal{T}_p\mathbb{N}\raisebox{.45ex}{\textup{\textbf{!`}}}\longrightarrow\mathcal{S}_p\mathbb{N}$
which induce a natural isomorphism of complexes
$(\mathcal{T}_{\bullet}\mathbb{N}\longrightarrow\mathbb{N}!)\raisebox{.45ex}{\textup{\textbf{!`}}}$
 and
$(\mathcal{S}_{\bullet}\mathbb{N}\longrightarrow\mathbb{N})$.

\end{theorem}
\begin{proof}
1. The left adjoint is the restriction of a functor
$\raisebox{.45ex}{\textup{\textbf{!`}}}:\mathbb{A}!$-bimod$\longrightarrow\mathbb{A}$-bimod.
For any $\mathbb{A}!$-bimodule $X$ and $i\in\mathcal{C}$ we define
$X\raisebox{.45ex}{\textup{\textbf{!`}}}^i$ as the colimit of a
particular functor over the poset $\mathcal{C}_i^{1}$ whose elements
are the 1-simplices of $\mathcal{C}_i=\{j | j\geq i\}$. The ordering
is $\sigma\ll\tau \Leftrightarrow\sigma=\tau$ or $\sigma$ is
degenerate ($d\sigma=c\sigma$), $\tau$ is not, and either
$d\sigma=d\tau$ or $d\sigma=c\tau$. We denote by
$X^{pq}=\varphi^{pp}X\varphi^{qq}$. Define
$F_{X}^i:\mathcal{C}_i^1\longrightarrow\mathbb{A}^i$-bimod on each
object $\sigma$ to be the coequalizer of the $\mathbb{A}^i$-bimodule
maps $X^{i,d\sigma}\otimes_k\mathbb{A}^{d\sigma}\rightrightarrows
X^{i,c\sigma}\otimes_{\mathbb{A}^{c\sigma}}\mathbb{A}^i$ given by
$x\otimes a\rightarrow
x\varphi^{d\sigma,c\sigma}\otimes\varphi^{i,d\sigma}(a)$ and
$x\otimes a\rightarrow xa\varphi^{d\sigma,c\sigma}\otimes1$. For
$\sigma\ll\tau$ in $\mathcal{C}_i^1$, the map
$F_{X}^i(\sigma\tau):F_{X}^i(\sigma)\rightarrow F_{X}^i(\tau)$ is
defined by $\overline{x\otimes
a}\rightarrow\overline{x\varphi^{c\sigma,c\tau}\otimes a}$.

Let $X\raisebox{.45ex}{\textup{\textbf{!`}}}^i:=colim F_{X}^i$,
$\forall i\in\mathcal{C}$ and $\iota_{\sigma}^i$  the canonical map
$F_{X}^i(\sigma)\rightarrow\
X\raisebox{.45ex}{\textup{\textbf{!`}}}^i$. To show that
$X\raisebox{.45ex}{\textup{\textbf{!`}}}$ is an
$\mathbb{A}$-bimodule, for each $h\leq i$ in $\mathcal{C}$, we have
to define a map
$T_{X\raisebox{.45ex}{\textup{\textbf{!`}}}}^{hi}:X\raisebox{.45ex}{\textup{\textbf{!`}}}
^i\rightarrow\ X\raisebox{.45ex}{\textup{\textbf{!`}}}^h$ such that
$T_{X\raisebox{.45ex}{\textup{\textbf{!`}}}}^{hj}=T_{X\raisebox{.45ex}{\textup{\textbf{!`}}}
}^{hi}T_{X\raisebox{.45ex}{\textup{\textbf{!`}}}}^{ij}$ if $h\leq
i\leq j$.

First, we have a natural transformation
$\Gamma_{X}^{hi}:F_{X}^i\rightarrow _{hi}|-|_{hi}\circ F_{X}^h\circ
inc_{hi}$, where
$inc_{hi}:\mathcal{C}_i^1\rightarrow\mathcal{C}_h^1$ is the
inclusion functor induced by $\mathcal{C}_i\subset\mathcal{C}_h$ and
$_{hi}|-|_{hi}:\mathbb{A}^h$-bimod$\rightarrow\mathbb{A}^i$-bimod is
the forgetful functor. To define $\Gamma_{X}^{hi}$ observe that, for
each $\sigma\in\mathcal{C}^1$, left multiplication by $\varphi^{hi}$
is an $\mathbb{A}^i$-$\mathbb{A}^{c\sigma}$ bimodule map
$:X^{i,c\sigma}\rightarrow X^{h,c\sigma}$, while
$\varphi^{hi}:\mathbb{A}^i\rightarrow\mathbb{A}^h$ is an
$\mathbb{A}^{c\sigma}$-$\mathbb{A}^i$ bimodule map. The map of
$\mathbb{A}^i$-bimodules
$(\varphi^{hi}\cdot-)\otimes\varphi^{hi}:X^{i,c\sigma}\otimes_{\mathbb{A}^{c\sigma}}
\mathbb{A}^i\rightarrow_{hi}|X^{h,c\sigma}\otimes_{\mathbb{A}^{c\sigma}}\mathbb{A}^i|_{hi}$
induces the $\mathbb{A}^i$-bimodule map
$(\Gamma_{X}^{hi})_{\sigma}:F_{X}^i(\sigma)\rightarrow_{hi}|F_{X}^h(\sigma)|_{hi}$
given by $\overline{x\otimes
a}\rightarrow\overline{\varphi^{hi}x\otimes\varphi^{hi}(a)}$.

Second, let $T_{X\raisebox{.45ex}{\textup{\textbf{!`}}}}^{hi}$ be
the composite of maps
$X\raisebox{.45ex}{\textup{\textbf{!`}}}^i=colim F_{X}^i\rightarrow
colim(_{hi}|-|_{hi}\circ F_{X}^h\circ
inc_{hi})=_{hi}|colim(F_{X}^h\circ inc_{hi})|_{hi}\rightarrow
_{hi}|colim
F_{X}^h|_{hi}=_{hi}|X\raisebox{.45ex}{\textup{\textbf{!`}}}^h|_{hi}$.
Thus we have
$T_{X\raisebox{.45ex}{\textup{\textbf{!`}}}}^{hi}(\iota_{\sigma}^i(\overline{x\otimes
a}))=\iota_{\sigma}^h(\overline{\varphi^{hi}x\otimes\varphi^{hi}(a)})$.

One may easily check now the identity
$T_{X\raisebox{.45ex}{\textup{\textbf{!`}}}}^{hj}=T_{X\raisebox{.45ex}{\textup{\textbf{!`}}}
}^{hi}T_{X\raisebox{.45ex}{\textup{\textbf{!`}}}}^{ij}$ for $h\leq
i\leq j$, so $X\raisebox{.45ex}{\textup{\textbf{!`}}}$ is an
$\mathbb{A}$-bimodule.

So far we have defined $\raisebox{.45ex}{\textup{\textbf{!`}}}$ on
the objects of $\mathbb{A}!$-bimod so we need to define it on  maps.
Let $g: X\rightarrow Y$ be an $\mathbb{A}!$-bimodule map. The
restriction of $g$ to $X^{ij}$ is an $\mathbb{A}^i$-$\mathbb{A}^j$
bimodule map, $g:X^{ij}\rightarrow Y^{ij}$, and for
$\sigma\in\mathcal{C}^1$, the $\mathbb{A}^i$-bimodule map $g\otimes
id:X^{i,c\sigma}\otimes_{\mathbb{A}^{c\sigma}}\mathbb{A}^i\rightarrow
Y^{i,c\sigma}\otimes_{\mathbb{A}^{c\sigma}}\mathbb{A}^i$ induces the
map $\widetilde{g}_{\sigma}^i:F_{X}^i(\sigma)\rightarrow
F_{Y}^i(\sigma)$ defined by $\overline{x\otimes
a}\rightarrow\overline{g(x)\otimes a}$. Its easy to check the
naturality of $\widetilde{g}_{\sigma}^i$ since $g$ is a
$\mathbb{A}!$-bimodule map and by taking the colimits we obtain an
$\mathbb{A}^i$-bimodule map
$g\raisebox{.45ex}{\textup{\textbf{!`}}}^i:X\raisebox{.45ex}{\textup{\textbf{!`}}}^i\rightarrow
Y\raisebox{.45ex}{\textup{\textbf{!`}}}^i$, given by
$g\raisebox{.45ex}{\textup{\textbf{!`}}}^i(\iota_{\sigma}^i(\overline{x\otimes
a}))=\iota_{\sigma}^i(\widetilde{g}_{\sigma}^i(\overline{x\otimes
a}))$. These are the components of an $\mathbb{A}$-bimodule map,
$i.e. T_{Y\raisebox{.45ex}{\textup{\textbf{!`}}}}^{hi}\circ
g\raisebox{.45ex}{\textup{\textbf{!`}}}^i=g\raisebox{.45ex}{\textup{\textbf{!`}}}^h\circ
T_{X\raisebox{.45ex}{\textup{\textbf{!`}}}}^{hi}$ for $h\leq i$
because of the commutative diagram
$$\xymatrix{F_{X}^i\ar[rr]^-{\Gamma_{X}^{hi}}\ar[d]^{\widetilde{g}^i}&&_{hi}|-|_{hi}\circ
F_{X}^h\circ inc_{hi}\ar[d]^{id\circ \widetilde{g}^h\circ{id}}\\
F_{Y}^i\ar[rr]^-{\Gamma_{Y}^{hi}}&&_{hi}|-|_{hi}\circ
F_{Y}^h\circ{inc}_{hi}}$$ In addition, one has
$(g_1g_2)\raisebox{.45ex}{\textup{\textbf{!`}}}=(g_1)\raisebox{.45ex}{\textup{\textbf{!`}}}\circ
(g_2)\raisebox{.45ex}{\textup{\textbf{!`}}}$ and
$(id)\raisebox{.45ex}{\textup{\textbf{!`}}}=id$ since
$\widetilde{g_1g_2}^i=\widetilde{g_1}^i\widetilde{g_2}^i$ and
$\widetilde{id}^i=id$, so $\raisebox{.45ex}{\textup{\textbf{!`}}}$
is a functor.

We now prove that the functor constructed above is a left adjoint to
$!$, when restricted to $\mathbb{A}$-albimod. Let $X$ be an aligned
$\mathbb{A}!$-bimodule. For $i\leq j$ we define
$\eta_{X}^{ij}:X^{ij}\rightarrow
(X\raisebox{.45ex}{\textup{\textbf{!`}}})!^{ij}$ to be the
$\mathbb{A}^i$-$\mathbb{A}^j$ bimodule map
$\eta_{X}^{ij}=\iota_{(j\leq j)}^i(x\otimes 1)\varphi^{ij}$. One may
check that for $h\leq i\leq j\leq q$, $a^h\in\mathbb{A}^h$,
$a^j\in\mathbb{A}^j$ and $x\in X^{ij}$ we have
$\eta_{X}^{hj}(a^h\varphi^{hi}\cdot x )= a^h\varphi^{hi}\cdot
\eta_{X}^{ij}(x)$ and $\eta_{X}^{iq}(x\cdot
a^j\varphi^{jq})=\eta_{X}^{ij}(x)\cdot a^j\varphi^{jq}$ so the
family of maps $\eta_{X}^{ij}$ determine an $\mathbb{A}!$-bimodule
natural map $\eta_{X}:X\rightarrow
(X\raisebox{.45ex}{\textup{\textbf{!`}}})!$.

Let $\mathbb{N}\in\mathbb{A}$-bimod. To define the components of the
counit
$\varepsilon_{\mathbb{N}}^i:(\mathbb{N}!)\raisebox{.45ex}{\textup{\textbf{!`}}}^i=colim
F_{\mathbb{N}!}^i\rightarrow\mathbb{N}^i$, we define a family of
$\mathbb{A}^i$-bimodule maps
$\varepsilon_{\mathbb{N}}^{i,\sigma}:F_{\mathbb{N}!}^i(\sigma)\rightarrow\mathbb{N}^i$
such that $\varepsilon_{\mathbb{N}}^{i,\tau}\circ
F_{\mathbb{N}!}^i(\sigma\tau)=\varepsilon_{\mathbb{N}}^{i,\sigma}$,
for $\sigma\ll\tau$ in $\mathcal{C}_i^1$ and use the universal
property of colimits. The $\mathbb{A}^i$-bilinear function
$\mathbb{N}^i\varphi^{i,c\sigma}\times\mathbb{A}^i\rightarrow\mathbb{N}^i$,
$(n\varphi^{i.c\sigma}, a)\rightarrow na$ is
$\mathbb{A}^{c\sigma}$-balanced and the induced
$\mathbb{A}^i$-bimodule map
$\mathbb{N}^i\varphi^{i,c\sigma}\otimes_{\mathbb{A}^{c\sigma}}\mathbb{A}^i
\rightarrow\mathbb{N}^i$ vanishes on $\{n\varphi^{d\sigma,
c\sigma}\otimes\varphi^{i.d\sigma}(a)-na\varphi^{d\sigma.c\sigma}\otimes
1 $\;$|$\;$ n\in\mathbb{N}^i\varphi^{i,d\sigma},
a\in\mathbb{A}^{d\sigma}\}$ so for each $\sigma\in\mathcal{C}_i^1$,
we obtain the $\mathbb{A}^i$-bimodule map
$\varepsilon_{\mathbb{N}}^{i,\sigma}:F_{\mathbb{N}!}^i(\sigma)\rightarrow\mathbb{N}^i$,
$\overline{n\varphi^{i,c\sigma}\otimes a}\rightarrow na$. We have
that $\varepsilon_{\mathbb{N}}^{i,\tau}\circ
F_{\mathbb{N}!}^i(\sigma\tau)(n\varphi^{i,c\sigma}\otimes
a)=\varepsilon_{\mathbb{N}}^{i,\tau}(n\varphi^{i,c\sigma}\cdot\varphi^{c\sigma,c\tau}\otimes
a)=na=\varepsilon_{\mathbb{N}}^{i,\sigma}(n\varphi^{i,c\sigma}\otimes
a)$ and thus the map $\varepsilon_{\mathbb{N}}^i$ is given by
$\varepsilon_{\mathbb{N}}^i(\iota_{\sigma}^i(\overline{n\varphi^{i,c\sigma}\otimes
a)})=na$. The maps $\varepsilon_{\mathbb{N}}^i$ determine a natural
map
$\varepsilon_{\mathbb{N}}:(\mathbb{N}!)\raisebox{.45ex}{\textup{\textbf{!`}}}\rightarrow\mathbb{N}$
of $\mathbb{A}$-bimodules since for $h\leq i$ and
$\sigma\in\mathcal{C}_i^1$ we have
$T_{\mathbb{N}}^{hi}\varepsilon_{\mathbb{N}}^{i,c\sigma}(\overline{n\varphi^{i,c\sigma}\otimes
a})=T_{\mathbb{N}}^{hi}(na)=T_{\mathbb{N}}^{hi}(n)\cdot
a=T_{\mathbb{N}}^{hi}(n)\varphi^{hi}(a)$ while
$\varepsilon_{\mathbb{N}}^{h,\sigma}(\Gamma_{\mathbb{N}!}^{hi})_{\sigma}
(\overline{n\varphi^{i,c\sigma}\otimes
a})=\varepsilon_{\mathbb{N}}^{h,\sigma}(\overline{\varphi^{hi}\cdot
n\varphi^{i,c\sigma}\otimes
\varphi^{hi}(a)})=\varepsilon_{\mathbb{N}}^{h,\sigma}(\overline
{T_{\mathbb{N}}^{hi}(n)\varphi^{h,c\sigma}\otimes\varphi^{hi}(a)})=T_{\mathbb{N}}^{hi}(n)
\varphi^{hi}(a).$

To finish the proof we show that $\eta$ and $\varepsilon$ form an
adjoint pair. To see that
$\varepsilon_{\mathbb{N}}!\circ\eta_{\mathbb{N}!}=id_{\mathbb{N}!} $
it is enough to check this on each $\mathbb{N}^i\varphi^{ij}$ and
$\varepsilon_{\mathbb{N}}!(\eta_{\mathbb{N}!}(n\varphi^{ij})=\varepsilon_{\mathbb{N}}!
(\iota_{(j\leq j)}^i(n\varphi^{ij}\otimes
1)\varphi^{ij})=\varepsilon_{\mathbb{N}}^i(\iota_{(j\leq
j)}^i(n\varphi^{ij}\otimes 1))\varphi^{ij}=(n\cdot
1)\varphi^{ij}=n\varphi^{ij}$, as required.

Last, for each $X\in\mathbb{A}$-albimod we need to verify that
$\varepsilon_{X\raisebox{.45ex}{\textup{\textbf{!`}}}}\circ\
\eta_{X}\raisebox{.45ex}{\textup{\textbf{!`}}}=id_{X\raisebox{.45ex}{\textup{\textbf{!`}}}}$.
This can be checked on a set of $\mathbb{A}^i$-bimodule generators
for each component $X\raisebox{.45ex}{\textup{\textbf{!`}}}^i$ and
the set $\{\iota_{(j\leq j)}^i(x\otimes 1)$\;$|$\;$ j\geq i, x\in
X^{ij}\}$ has this property. Since
$(\varepsilon_{X\raisebox{.45ex}{\textup{\textbf{!`}}}}^i\circ
\eta_{X}\raisebox{.45ex}{\textup{\textbf{!`}}}^i)(\iota_{(j\leq
j)}^i(x\otimes
1)))=\varepsilon^i_{X\raisebox{.45ex}{\textup{\textbf{!`}}}}(\iota_{(j\leq
j)}^i(\eta_{X}(x)\otimes
1))=\varepsilon_{X\raisebox{.45ex}{\textup{\textbf{!`}}}}^i(\iota_{(j\leq
j)}^i(\iota_{(j\leq j)}^i(x\otimes 1)\varphi^{ij}\otimes
1))=\iota_{(j\leq j)}^i(x\otimes 1)\cdot 1$ we obtain the required
identity.

2. Because both $\mathcal{T}_p\mathbb{N}$ and
$\mathcal{S}_p\mathbb{N}$ are coproducts and
$\raisebox{.45ex}{\textup{\textbf{!`}}}$, as a left adjoint,
preserves colimits it is enough to find natural isomorphisms
$\gamma^{\sigma}:(\mathcal{T}_p^{\sigma})\raisebox{.45ex}{\textup{\textbf{!`}}}\longrightarrow\mathcal{S}_p^{\sigma}$
such that, for $0\leq r\leq p$, the following square
$$\xymatrix{(\mathcal{T}_{p}^{\sigma}\mathbb{N})\raisebox{.45ex}{\textup{\textbf{!`}}}\ar[r]^{(d_r^{\mathcal{T},\sigma})
\raisebox{.45ex}{\textup{\textbf{!`}}}}
\ar[d]_{\gamma_{\mathbb{N}}^{\sigma}}&(\mathcal{T}_{p-1}^{\sigma_r}\mathbb{N})\raisebox{.45ex}{\textup{\textbf{!`}}}\ar[d]^
{\gamma_{\mathbb{N}}^{\sigma_r}}\\
\mathcal{S}_p^{\sigma}\mathbb{N}\ar[r]^{d_{r}^{\mathcal{S},\sigma}}&
\mathcal{S}_{p-1}^{\sigma_r}\mathbb{N}}$$ commutes, where, when
$p=0$, we interpret the right column as the counit
$\varepsilon_\mathbb{N}:(\mathbb{N}!)\raisebox{.45ex}{\textup{\textbf{!`}}}\longrightarrow\mathbb{N}$
and $d_0^{\mathcal{T}}$ and $d_0^{\mathcal{S}}$ as the
augmentations. To construct the isomorphisms, for $p>0$, observe
that, for each $\sigma\in\mathcal{C}^{[p]}$, the diagram
$$\xymatrix{\mathbb{A}-{\mathrm{bimod}}\ar[r]^{!}\ar[d]_{(d\sigma)^{*}}&
\mathbb{A}!-\mathrm{albimod}\ar[d]^{(-)^{d\sigma,c\sigma}}\\
\mathbb{A}^{d\sigma}\mathrm{-bimod}\ar[d]_{_{d\sigma,c\sigma}|-|_{d\sigma,c\sigma}}&
\mathbb{A}^{d\sigma}-\mathrm{mod}-\mathbb{A}^{c\sigma}\ar[d]^{_{d\sigma,c\sigma}|-|}\\
\mathbb{A}^{c\sigma}-\mathrm{bimod}\ar[r]^{id}&\mathbb{A}^{c\sigma}-\mathrm{bimod}}$$
is commutative. Since each functor in it admits a left adjoint and
$_{d\sigma,c\sigma}|-|_{d\sigma,c\sigma}\circ(d\sigma)^{*}=_{d\sigma,c\sigma}|-|
\circ(-)^{d\sigma,c\sigma}\circ{!}$ we have the isomorphisms,
natural in $N$,
$$\gamma_{N}^{\sigma}:(L_{d\sigma,c\sigma}(\mathbb{A}^{d\sigma}
\otimes_{\mathbb{A}^{c\sigma}}N))\raisebox{.45ex}{\textup{\textbf{!`}}}\longrightarrow(d\sigma)_{!}(\mathbb{A}^{d\sigma}
\otimes_{\mathbb{A}^{c\sigma}}N\otimes_{\mathbb{A}^{c\sigma}}\mathbb{A}^{d\sigma}).$$

The $\mathbb{A}^i$ bimodule
$((L_{d\sigma,c\sigma}(\mathbb{A}^{d\sigma}
\otimes_{\mathbb{A}^{c\sigma}}N))
)\raisebox{.45ex}{\textup{\textbf{!`}}}^i$ is generated by
$\{\iota_{(j\leq j)}^i((1\otimes n)\otimes 1)$\;$|$\;$j\geq c\sigma,
n\in N\}$ for $i\leq d\sigma$ and is $0$ if $i\nleq d\sigma$.
Tracing through the adjunction we obtain that, for each
$i\in\mathcal{C}$, $\gamma_{N}^{\sigma,i}(\iota_{(j\leq
j)}^i((1\otimes n)\otimes 1))=1\otimes n\otimes1.$

For all $\mathbb{N}$ and $\sigma$ we define
$\gamma_{\mathbb{N}}^{\sigma}=\gamma_{\mathbb{N}^{c\sigma}}^{\sigma}$
and, for $p>0$, we have that
$(d_r^{\mathcal{T},\sigma})\raisebox{.45ex}{\textup{\textbf{!`}}}^i(\iota_{(j\leq
j)}^i(\\(1\otimes n)\otimes 1))=\iota_{(j\leq
j)}^i(d_r^{\mathcal{T},\sigma}(1\otimes n)\otimes 1)=\iota_{(j\leq
j)}^i((1\otimes T_{\mathbb{N}}^{c\sigma_r,d\sigma}(n))\otimes 1)$,
while $d_r^{\mathcal{S},\sigma}(1\otimes n\otimes 1)=1\otimes
T_{\mathbb{N}}^{c\sigma_r,c\sigma}(n)\otimes 1$, so the square is
commutative. When $p=0$, we obtain, for $\sigma\in\mathcal{C}^{[0]}$
and $i\leq d\sigma=c\sigma\leq j$ that
$\varepsilon_{\mathbb{N}}^i\circ(\varepsilon^{\mathcal{T},
\sigma})\raisebox{.45ex}{\textup{\textbf{!`}}}^i(\iota_{(j\leq
j)}^i((1\otimes n)\otimes
1))=\varepsilon_{\mathbb{N}}^i(\iota_{(j\leq
j)}^i(\varepsilon_{ij}^{\mathcal{T},\sigma}(1\otimes n)\otimes
1))=\varepsilon_{\mathbb{N}}^i(\iota_{(j\leq
j)}^i(T_{\mathbb{N}}^{i,c\sigma}(n)\varphi^{ij}\otimes
1))=T_{\mathbb{N}}^{i,c\sigma}(n)=\varepsilon^{\mathcal{S},\sigma,i}(1\otimes
n\otimes
1)=\varepsilon^{\mathcal{S},\sigma,i}\circ\gamma_{\mathbb{N}}^{\sigma,i}(\iota_{(j\leq
j)}^i((1\otimes n)\otimes 1)),$ so the square commutes in this case
too.
\end{proof}


\begin{thebibliography}{}
\bibitem{1} M. Gerstenhaber and S. D. Schack, ``Algebraic Cohomology and Deformation Theory'',
Deformation Theory of Algebras and Structures and Applications,
Kluwer, Dordrecht (1988) 11-264.
\bibitem{2} M. Gerstenhaber and S. D. Schack, ``The Cohomology of Presheaves of Algebras:
 Presheaves over a Partially Ordered Set'', Trans. Amer. Math. Soc.
 310 (1988) 135-165.
\bibitem{3} C. B. Kullmann, ``Adjoints and Cohomology for Presheaves
of Algebras Over a Poset '', SUNY at Buffalo Ph.D. thesis, (1998).
\bibitem{4} S. MacLane, ``Homology '', Spinger-Verlag, Berlin, (1967).
\bibitem{5}S. I. Gelfand and Yu. I. Manin, ``Methods of Homological
Algebra'', Springer Verlag (1996).
\bibitem{6} B. Keller, ``Hochschild Cohomology and the Derived Picard Group'',
Journal of Pure and Applied Algebra 190 (2004), 177-196.
\bibitem{7} Alin A. Stancu, ``Hochschild Coheomology and Derived
Categories'', SUNY at Buffalo Ph.D. thesis, (2006).
\bibitem{8} Wendy Lowen
and Michel Van Den Bergh ``A Hochschild Cohomology Comparison
Theorem for Prestaks'', arXiv:0905.2354v1 [math.KT]
\end{thebibliography}
\end{document}